\documentclass[12pt]{amsart}
\usepackage{amsmath}
\usepackage{amsxtra}
\usepackage{amscd}
\usepackage{amsthm}
\usepackage{amsfonts}
\usepackage{amssymb}
\usepackage{eucal}
\usepackage{verbatim}
\usepackage[all]{xy}
\usepackage{color}
\definecolor{deepjunglegreen}{rgb}{0.0, 0.29, 0.29}
\definecolor{darkspringgreen}{rgb}{0.09, 0.45, 0.27}

\usepackage{stmaryrd}

\makeatletter
\usepackage{etex,etoolbox,needspace}
\pretocmd\section{\Needspace*{4\baselineskip}}{}{}
\makeatletter

\usepackage[pdftex,bookmarks=false,colorlinks=true,citecolor=darkspringgreen,debug=true,
  naturalnames=true,pdfnewwindow=true]{hyperref}

\usepackage{xcolor}
\usepackage{amscd}

\newtheorem{thm}{Theorem}[subsection]

\newtheorem{cor}[thm]{Corollary}

\newtheorem{lem}[thm]{Lemma}
\newtheorem{prop}[thm]{Proposition}
\newtheorem{conj}[thm]{Conjecture}

\theoremstyle{definition}

\theoremstyle{remark}
\newtheorem{rem}[thm]{Remark}

\newcommand{\nc}{\newcommand}
\nc{\renc}{\renewcommand} \nc{\ssec}{\subsection}
\nc{\sssec}{\subsubsection}

\nc{\on}{\operatorname} \nc{\wh}{\widehat}
\nc\ol{\overline} \nc\ul{\underline} \nc\wt{\widetilde}

\newcommand{\red}[1]{{\color{red}#1}}

\newcommand{\Nt}{{\widetilde{\mathcal N}}}
\newcommand{\gt}{{\widetilde{\Lg}}}

\newcommand{\F}{{\mathcal F}}

\renewcommand{\P}{{\mathcal P}}

\newcommand{\N}{{\mathcal N}}

\renewcommand{\O}{{\mathcal O}}

\newcommand{\Ce}{{\mathbb C}}
\newcommand{\Zet}{{\mathbb Z}}

\newcommand{\Gm}{{\mathbb G}_m}

\newcommand{\oplusl}{\bigoplus\limits}

\newcommand{\suml}{\sum\limits}

\newcommand{\bs}{\backslash}

\newcommand{\LG}{{G^\vee}}

\newcommand{\Lg}{{\mathfrak g}^\svee}

\newcommand{\Ap}{\bar{A}}
\newcommand{\Hop}{\on{Ho}(\bA_V\modu_{\on{perv}}^{\on{fr}})}

\newcommand{\sH}{{\sf H}}

\newcommand{\St}{{\mathbf{St}}}

\newcommand\alp{\alpha}		
		
		\newcommand\Gam{\Gamma}
\newcommand\del{\delta}		
\newcommand\eps{\varepsilon}

\newcommand\lam{\lambda}		\newcommand\Lam{\Lambda}
\newcommand\sig{\sigma}

\newcommand\calF{{\mathcal{F}}}
\newcommand\calG{{\mathcal{G}}}
\newcommand\calH{{\mathcal{H}}}

\newcommand\calK{{\mathcal{K}}}

\newcommand\calO{{\mathcal{O}}}

\newcommand\calS{{\mathcal{S}}}

\newcommand\bfq{{\mathbf q}}		
		
		\newcommand\bfS{{\mathbf S}}

\newcommand\QQ{\mathbb{Q}}
\newcommand\WW{\mathbb{W}}
\newcommand\EE{\mathbb{E}}

\newcommand\II{\mathbb{I}}

\newcommand\PP{\mathbb{P}}
\renewcommand\AA{\mathbb{A}}

\newcommand\FF{\mathbb{F}}
\newcommand\GG{\mathbb{G}}

\newcommand\KK{\mathbb{K}}

\newcommand\ZZ{\mathbb{Z}}

\newcommand\CC{\mathbb{C}}

	\newcommand\grb{{\mathfrak{b}}}

	\newcommand\grg{{\mathfrak{g}}}

\newcommand\sdp{\times \hskip -0.3em {\raise 0.3ex
\hbox{$\scriptscriptstyle |$}}} 

\newcommand\codim{\operatorname{codim}}

\newcommand\PGL{{\rm PGL}}

\newcommand\Perv{\operatorname{Perv}}

\newcommand\sgn{\operatorname{sgn}}

\newcommand\spec{{\rm spec}}

\newcommand\oG{{\overline{G}}}

\newcommand\oX{{\overline{X}}}

\newcommand\hatZ{{\widehat{Z}}}

\newcommand\x{\times}

\newcommand\qlb{{\overline \QQ}_l}

\nc\LGr{{\mathbf LG}_{\rho}}

\nc{\BA}{{\mathbb{A}}} \nc{\BC}{{\mathbb{C}}} \nc{\BF}{{\mathbb{F}}}
\nc{\BD}{{\mathbb{D}}} \nc{\BG}{{\mathbb{G}}} \nc{\BQ}{{\mathbb{Q}}}
\nc{\BM}{{\mathbb{M}}} \nc{\BN}{{\mathbb{N}}} \nc{\BO}{{\mathbb{O}}}
\nc{\BP}{{\mathbb{P}}} \nc{\BR}{{\mathbb{R}}}
\nc{\BZ}{{\mathbb{Z}}} \nc{\BS}{{\mathbb{S}}} \nc{\BW}{{\mathbb{W}}}

\nc{\CA}{{\mathcal{A}}} \nc{\CB}{{\mathcal{B}}} \nc{\CalC}{{\mathcal{C}}} \nc{\CalD}{{\mathcal{D}}}
\nc{\CE}{{\mathcal{E}}} \nc{\CF}{{\mathcal{F}}}
\nc{\CG}{{\mathcal{G}}} \nc{\CH}{{\mathcal{H}}}
\nc{\CI}{{\mathcal{I}}} \nc{\CK}{{\mathcal{K}}} \nc{\CL}{{\mathcal{L}}}
\nc{\CM}{{\mathcal{M}}} \nc{\CN}{{\mathcal{N}}}
\nc{\CO}{{\mathcal{O}}} \nc{\CP}{{\mathcal{P}}}
\nc{\CQ}{{\mathcal{Q}}} \nc{\CR}{{\mathcal{R}}}
\nc{\CS}{{\mathcal{S}}} \nc{\CT}{{\mathcal{T}}}
\nc{\CU}{{\mathcal{U}}} \nc{\CV}{{\mathcal{V}}}  \nc{\CY}{{\mathcal Y}}
\nc{\CW}{{\mathcal{W}}} \nc{\CZ}{{\mathcal{Z}}}

\nc{\cM}{{\check{\mathcal M}}{}} \nc{\csM}{{\check{\mathcal A}}{}}
\nc{\obM}{{\overset{\circ}{\mathbf M}}{}}
\nc{\oCA}{{\overset{\circ}{\mathcal A}}{}}
\nc{\obA}{{\overset{\circ}{\mathbf A}}{}}
\nc{\ooM}{{\overset{\circ}{M}}{}}
\nc{\osM}{{\overset{\circ}{\mathsf M}}{}}
\nc{\vM}{{\overset{\bullet}{\mathcal M}}{}}
\nc{\nM}{{\underset{\bullet}{\mathcal M}}{}}
\nc{\obD}{{\overset{\circ}{\mathbf D}}{}}
\nc{\cp}{{\overset{\circ}{\mathbf p}}{}}
\nc{\ofZ}{{\overset{\circ}{\mathfrak Z}}{}}

\nc{\ff}{{\mathfrak{f}}} \nc{\fv}{{\mathfrak{v}}}
\nc{\fa}{{\mathfrak{a}}} \nc{\fb}{{\mathfrak{b}}}
\nc{\fd}{{\mathfrak{d}}} \nc{\fe}{{\mathfrak{e}}}
\nc{\fg}{{\mathfrak{g}}} \nc{\fgl}{{\mathfrak{gl}}}
\nc{\fh}{{\mathfrak{h}}} \nc{\fri}{{\mathfrak{i}}}
\nc{\fj}{{\mathfrak{j}}} \nc{\fk}{{\mathfrak{k}}} \nc{\fl}{{\mathfrak{l}}}
\nc{\fm}{{\mathfrak{m}}} \nc{\fn}{{\mathfrak{n}}}
\nc{\ft}{{\mathfrak{t}}} \nc{\fu}{{\mathfrak{u}}}
\nc{\fw}{{\mathfrak{w}}} \nc{\fz}{{\mathfrak{z}}}
\nc{\fp}{{\mathfrak{p}}} \nc{\fq}{{\mathfrak{q}}} \nc{\frr}{{\mathfrak{r}}}
\nc{\fs}{{\mathfrak{s}}} \nc{\fsl}{{\mathfrak{sl}}}
\nc{\fso}{{\mathfrak{so}}} \nc{\fsp}{{\mathfrak{sp}}} \nc{\osp}{{\mathfrak{osp}}}
\nc{\hsl}{{\widehat{\mathfrak{sl}}}}
\nc{\hgl}{{\widehat{\mathfrak{gl}}}}
\nc{\hg}{{\widehat{\mathfrak{g}}}}
\nc{\chg}{{\widehat{\mathfrak{g}}}{}^\vee}
\nc{\hn}{{\widehat{\mathfrak{n}}}}
\nc{\chn}{{\widehat{\mathfrak{n}}}{}^\vee}

\nc{\fA}{{\mathfrak{A}}} \nc{\fB}{{\mathfrak{B}}} \nc{\fC}{{\mathfrak{C}}}
\nc{\fD}{{\mathfrak{D}}} \nc{\fE}{{\mathfrak{E}}}
\nc{\fF}{{\mathfrak{F}}} \nc{\fG}{{\mathfrak{G}}} \nc{\fH}{{\mathfrak{H}}}
\nc{\fI}{{\mathfrak{I}}} \nc{\fJ}{{\mathfrak{J}}}
\nc{\fK}{{\mathfrak{K}}} \nc{\fL}{{\mathfrak{L}}}
\nc{\fM}{{\mathfrak{M}}} \nc{\fN}{{\mathfrak{N}}}
\nc{\frP}{{\mathfrak{P}}} \nc{\fQ}{{\mathfrak{Q}}}
\nc{\fS}{{\mathfrak{S}}} \nc{\fT}{{\mathfrak{T}}} \nc{\fU}{{\mathfrak{U}}}
\nc{\fV}{{\mathfrak{V}}} \nc{\fW}{{\mathfrak{W}}}
\nc{\fX}{{\mathfrak{X}}} \nc{\fY}{{\mathfrak{Y}}}
\nc{\fZ}{{\mathfrak{Z}}}

\nc{\ba}{{\mathbf{a}}}
\nc{\bb}{{\mathbf{b}}} \nc{\bc}{{\mathbf{c}}} \nc{\be}{{\mathbf{e}}}
\nc{\bg}{{\mathbf{g}}} \nc{\bj}{{\mathbf{j}}} \nc{\bm}{{\mathbf{m}}}
\nc{\bn}{{\mathbf{n}}} \nc{\bp}{{\mathbf{p}}}
\nc{\bq}{{\mathbf{q}}} \nc{\br}{{\mathbf{r}}} \nc{\bt}{{\mathbf{t}}}
\nc{\bv}{{\mathbf{v}}}
\nc{\bx}{{\mathbf{x}}} \nc{\by}{{\mathbf{y}}} \nc{\bz}{{\mathbf{z}}}
\nc{\bw}{{\mathbf{w}}} \nc{\bA}{{\mathbf{A}}}
\nc{\bB}{{\mathbf{B}}} \nc{\bC}{{\mathbf{C}}}
\nc{\bD}{{\mathbf{D}}} \nc{\bF}{{\mathbf{F}}} \nc{\bG}{{\mathbf{G}}}
\nc{\bH}{{\mathbf{H}}} \nc{\bI}{{\mathbf{I}}} \nc{\bJ}{{\mathbf{J}}}
\nc{\bK}{{\mathbf{K}}} \nc{\bM}{{\mathbf{M}}} \nc{\bN}{{\mathbf{N}}}
\nc{\bO}{{\mathbf{O}}} \nc{\bS}{{\mathbf{S}}} \nc{\bT}{{\mathbf{T}}}
\nc{\bU}{{\mathbf{U}}} \nc{\bV}{{\mathbf{V}}} \nc{\bW}{{\mathbf{W}}}
\nc{\bX}{{\mathbf{X}}}
\nc{\bY}{{\mathbf{Y}}} \nc{\bP}{{\mathbf{P}}}
\nc{\bZ}{{\mathbf{Z}}} \nc{\bh}{{\mathbf{h}}}

\nc{\sA}{{\mathsf{A}}} \nc{\sfb}{{\mathsf{b}}} \nc{\sB}{{\mathsf{B}}}
\nc{\sC}{{\mathsf{C}}} \nc{\sD}{{\mathsf{D}}}
\nc{\sE}{{\mathsf{E}}} \nc{\sF}{{\mathsf{F}}} \nc{\sG}{{\mathsf{G}}}
\nc{\sI}{{\mathsf{I}}} \nc{\sK}{{\mathsf{K}}} \nc{\sk}{{\mathsf{k}}} \nc{\sL}{{\mathsf{L}}}
\nc{\sfm}{{\mathsf{m}}} \nc{\sM}{{\mathsf{M}}} \nc{\sN}{{\mathsf{N}}}
\nc{\sO}{{\mathsf{O}}} \nc{\sQ}{{\mathsf{Q}}} \nc{\sP}{{\mathsf{P}}}
\nc{\sT}{{\mathsf{T}}}
\nc{\sV}{{\mathsf{V}}} \nc{\sW}{{\mathsf{W}}}
\nc{\sfp}{{\mathsf{p}}} \nc{\sq}{{\mathsf{q}}} \nc{\sr}{{\mathsf{r}}}
\nc{\sfs}{{\mathsf{s}}} \nc{\sfS}{{\mathsf{S}}} \nc{\st}{{\mathsf{t}}}
\nc{\sfc}{{\mathsf{c}}} \nc{\sd}{{\mathsf{d}}} \nc{\sx}{{\mathsf{x}}}
\nc{\sy}{{\mathsf{y}}} \nc{\sz}{{\mathsf{z}}} \nc{\sZ}{{\mathsf{Z}}}

\nc{\tA}{{\widetilde{\mathbf{A}}}}
\nc{\tB}{{\widetilde{\mathcal{B}}}}
\nc{\tg}{{\widetilde{\mathfrak{g}}}} \nc{\tG}{{\widetilde{G}}}
\nc{\TM}{{\widetilde{\mathbb{M}}}{}}
\nc{\tO}{{\widetilde{\mathsf{O}}}{}}
\nc{\tU}{{\widetilde{\mathfrak{U}}}{}} \nc{\TZ}{{\tilde{Z}}}
\nc{\tx}{{\tilde{x}}} \nc{\tbv}{{\tilde{\bv}}}
\nc{\tfP}{{\widetilde{\mathfrak{P}}}{}} \nc{\tz}{{\tilde{\zeta}}}
\nc{\tmu}{{\tilde{\mu}}}

\nc{\urho}{\underline{\rho}} \nc{\uB}{\underline{B}}
\nc{\uC}{{\underline{\mathbb{C}}}} \nc{\ui}{\underline{i}}
\nc{\ofP}{{\overline{\mathfrak{P}}}}

\nc{\hrho}{{\hat{\rho}}}
\nc{\blambda}{{\boldsymbol{\lambda}}} \nc{\bmu}{{\boldsymbol{\mu}}} \nc{\bnu}{{\boldsymbol{\nu}}}
\nc{\btheta}{{\boldsymbol{\theta}}} \nc{\bzeta}{{\boldsymbol{\zeta}}} \nc{\bta}{{\boldsymbol{\eta}}}

\nc{\one}{{\mathbf{1}}} \nc{\two}{{\mathbf{t}}}

\nc{\Sym}{\mathop{\operatorname{\rm Sym}}}
\nc{\Tot}{{\mathop{\operatorname{\rm Tot}}}}
\nc{\Spec}{\mathop{\operatorname{\rm Spec}}}
\nc{\Ker}{{\mathop{\operatorname{\rm Ker}}}}
\nc{\Isom}{{\mathop{\operatorname{\rm Isom}}}}
\nc{\Hilb}{{\mathop{\operatorname{\rm Hilb}}}}
\nc{\deeq}{{\mathop{\operatorname{\rm deeq}}}}
\nc{\End}{{\mathop{\operatorname{\rm End}}}}
\nc{\Ext}{{\mathop{\operatorname{\rm Ext}}}}
\nc{\Hom}{{\mathop{\operatorname{\rm Hom}}}}
\nc{\CHom}{{\mathop{\operatorname{{\mathcal{H}}\it om}}}}
\nc{\GL}{{\mathop{\operatorname{\rm GL}}}}
\nc{\SL}{{\mathop{\operatorname{\rm SL}}}}
\nc{\SO}{{\mathop{\operatorname{\rm SO}}}}
\nc{\Sp}{{\mathop{\operatorname{\rm Sp}}}}
\nc{\OSp}{{\mathop{\operatorname{\rm SOSp}}}}
\nc{\gr}{{\mathop{\operatorname{\rm gr}}}}
\nc{\Id}{{\mathop{\operatorname{\rm Id}}}}
\nc{\perf}{{\mathop{\operatorname{\rm perf}}}}
\nc{\defi}{{\mathop{\operatorname{\rm def}}}}
\nc{\length}{{\mathop{\operatorname{\rm length}}}}
\nc{\supp}{{\mathop{\operatorname{\rm supp}}}}
\nc{\HC}{{\mathcal H}{\mathcal C}}

\nc{\pr}{{\operatorname{pr}}}
\nc{\Cliff}{{\mathsf{Cliff}}}
\nc{\loc}{{\operatorname{loc}}}
\nc{\Fl}{{\mathcal{F}\ell}} \nc{\Ffl}{{\mathbf{Fl}}}
\nc{\Fib}{{\mathsf{Fib}}}
\nc{\Coh}{{\operatorname{Coh}}} \nc{\FCoh}{{\mathsf{FCoh}}}
\nc{\Perf}{{\mathsf{Perf}}}
\nc{\wtimes}{\mathbin{\widetilde\times}}

\nc{\reg}{{\text{\rm reg}}}
\nc{\self}{{\text{\rm self}}}
\nc{\gvee}{{\mathfrak g}^{\!\scriptscriptstyle\vee}}
\nc{\tvee}{{\mathfrak t}^{\!\scriptscriptstyle\vee}}
\nc{\nvee}{{\mathfrak n}^{\!\scriptscriptstyle\vee}}
\nc{\bvee}{{\mathfrak b}^{\!\scriptscriptstyle\vee}}
       \nc{\rhovee}{\rho^{\!\scriptscriptstyle\vee}}

\nc{\cplus}{{\mathbf{C}_+}} \nc{\cminus}{{\mathbf{C}_-}}
\nc{\cthree}{{\mathbf{C}_*}} \nc{\Qbar}{{\bar{Q}}}

\newcommand\iso{\mathbin{\vphantom{j^{X^2}}\smash{\overset{\sim}{\vphantom{\rule{0pt}{0.20em}}\smash{\longrightarrow}}}}}
\nc{\Gtimes}{\vphantom{j^{X^2}}\smash{\overset{G}{\vphantom{\rule{0pt}{0.30em}}\smash{\times}}}}
\nc{\sGtimes}{\vphantom{j^{X^2}}\smash{\overset{\mathsf G}{\vphantom{\rule{0pt}{0.30em}}\smash{\times}}}}

\nc{\svee}{{\!\scriptscriptstyle\vee}}

\nc{\bOmega}{{\overline{\Omega}}}

\nc{\seq}[1]{\stackrel{#1}{\sim}}

\nc{\aff}{{\operatorname{aff}}}
\nc{\can}{{\operatorname{can}}}
\nc{\var}{{\operatorname{var}}}
\nc{\fin}{{\operatorname{fin}}}
\nc{\mir}{{\operatorname{mir}}}
\nc{\triv}{{\operatorname{triv}}}
\nc{\ext}{{\operatorname{ext}}}
\nc{\righ}{{\operatorname{right}}}
\nc{\lef}{{\operatorname{left}}}
\nc{\forg}{{\operatorname{forg}}}
\nc{\fid}{{\operatorname{fd}}}
\nc{\Fr}{{\operatorname{Fr}}}
\nc{\odd}{{\operatorname{odd}}}
\nc{\even}{{\operatorname{even}}}
\nc{\modu}{{\operatorname{-mod}}}
\nc{\Gr}{{\mathbf{Gr}}}
\nc{\FT}{{\operatorname{FT}}}
\nc{\Mat}{{\operatorname{Mat}}}
\nc{\MSt}{{\operatorname{MSt}}}
\nc{\sph}{{\operatorname{sph}}}
\nc{\GR}{{\mathbf{Gr}}}
\nc{\Rep}{{\operatorname{Rep}}}
\nc{\IC}{{\operatorname{IC}}}
\nc{\Bun}{{\operatorname{Bun}}}
\nc{\Stab}{{\operatorname{Stab}}}
\nc{\pt}{{\operatorname{pt}}}
\nc{\bfmu}{{\boldsymbol{\mu}}}
\nc{\bfomega}{{\boldsymbol{\omega}}}
\nc{\dslash}{/\!\!/}
\nc{\tslash}{/\!\!/\!\!/}
\nc\QCoh{\operatorname{QCoh}}
\nc\IndCoh{\operatorname{IndCoh}}
\nc\Maps{\operatorname{Maps}}
\nc\Dmod{D-\operatorname{mod}}
\newcommand\Hecke{\operatorname{Hecke}}
\renewcommand{\AA}{\mathbb A}
\nc\Av{\operatorname{Av}}
\nc\LS{\operatorname{LocSys}}

\newcommand*\circled[1]
{*{#1}}
\makeatletter
\newcommand{\raisemath}[1]{\mathpalette{\raisem@th{#1}}}
\newcommand{\raisem@th}[3]{\raisebox{#1}{$#2#3$}}
\nc{\binlim}[2][]{\def\@tempa{#1}\@ifnextchar^{\@binlim{#2}}{\@binlim{#2}^{}}}
\def\@binlim#1^#2{\mathbin{\@ifempty{#2}{\mathop{#1}}{\mathop{#1}\@xp\displaylimits\@tempa^{#2}}}}

\makeatother

\nc\cX{{\mathcal X}}

\nc{\isoto}
    {\stackrel{\displaystyle\raisemath{-.2ex}{\sim}}\to}
\nc\cD\CalD
\nc\setm\setminus
\nc\cE\CE
\renc\Hecke{\mathit{\CH\kern-.2ex ecke}}
\nc\bsl\backslash
\nc\bGO{{\bG_\bO}}
\nc\bla\blambda
\nc\blt\bullet
\nc\opp{{\on{op}}}
\nc\srel\stackrel
\nc\tbx{\binlim{\widetilde\boxtimes{}}}
\nc\into\hookrightarrow
\nc\otto\leftrightarrow
\renc\setminus\smallsetminus
\nc\phitau{\varphi\tau}
\newenvironment{i-ii-iii}{%
\begin{enumerate}
}%
{\end{enumerate}}
\nc\ceil[1]{\lceil#1\rceil}  \nc\floor[1]{\lfloor#1\rfloor}
\nc\Lie{\on{Lie}}
\nc\sS{{\mathsf S}}
\nc\vvv{\ensuremath{\red\surd}}
\def\arxiv#1{\href{http://arxiv.org/abs/#1}{\tt arXiv:#1}} 
\nc\sTr{\operatorname{sTr}}
\nc\Whit{\operatorname{Whit}}
\nc\KL{\operatorname{KL}}

\makeatletter
\renewcommand{\subsection}{\@startsection{subsection}{2}{0pt}{-3ex
plus -1ex minus -0.2ex}{-2mm plus -0pt minus
-2pt}{\normalfont\bfseries}} \makeatother

\numberwithin{equation}{subsection}
\allowdisplaybreaks
\nc\mto{\mapsto }
\nc\en{\enspace }

\nc{\XF}{{\mathfrak F}}

\begin{document}

\title{A fusion construction of  local $L$-factors}
\author[R.~Bezrukavnikov]{Roman Bezrukavnikov}
\address{Department of Mathematics, Massachusetts Institute of Technology,
  Cambridge, MA, 02139, USA}
\email{bezrukav@math.mit.edu}

\author[A.~Braverman]{Alexander Braverman}
\address{Department of Mathematics, University of Toronto and Perimeter Institute
of Theoretical Physics, Waterloo, Ontario, Canada, N2L 2Y5}
\email{braval@math.toronto.edu}

\author[M.~Finkelberg]{Michael Finkelberg}
\address{Department of Mathematics\\
National Research University Higher School of Economics\\
Russian Federation, Usacheva st.\ 6, 119048, Moscow\\
\newline Skolkovo Institute of Science and Technology\\
\newline Institute for Information Transmission Problems of RAS}
\email{fnklberg@gmail.com}

\author[D.~Kazhdan]{David Kazhdan}
\address{Institute of Mathematics\\
The Hebrew University of Jerusalem\\
Givat-Ram, Jerusalem,  91904\\
Israel}
\email{kazhdan@math.huji.ac.il}

\dedicatory{To the memory of Yuri Ivanovich Manin}

\maketitle

\begin{abstract} We propose a new conjectural way to calculate the local $L$-factor
$L=L_\chi(\pi,\rho,s)$ where $\pi$ is a representation of a $p$-adic group $G$, $\rho$ is
an algebraic representation of the dual group $\LG$ and $\chi$ is an algebraic character
of $G$ satisfying a positivity condition, and prove its validity in some cases.

A method  going back to Godement and Jacquet \cite{gj} yields a description of $L$
using as an input a certain space ${\mathcal S}_\rho$ of functions on $G$ depending on $\rho$.
A (partly conjectural) description  of ${\mathcal S}_\rho$ involving trace of Frobenius
functions associated to perverse sheaves on the loop space of a semigroup containing $G$ was developed
in \cite{BNS}, partly building on \cite{bk1}, \cite{bk2}.
Here we propose a different, more general conjectural
description of ${\mathcal S}_\rho$: it also refers to trace of Frobenius functions but
instead of the loop space of a semi-group we work with the ramified global Grassmannian
fibering over the configuration space of points on a global curve defined by
Beilinson-Drinfeld and Gaitsgory.

Our main result asserts validity of our conjectures
where $\pi$ is generated by an Iwahori fixed vector: we show that in this case it is compatible
with the standard formula for $L$ involving local Langlands correspondence.
 The proof is based on properties
of the coherent realization of the affine Hecke category.
\end{abstract}

\tableofcontents

\section{Introduction}
\subsection{Local $L$-factors}
$L$-functions are central to the theory of automorphic forms, in particular, a large part of Langlands
conjectures is known to follow from conjectural properties of appropriate $L$-functions
(such as analytic continuation and functional equation).
Global $L$-functions
are defined as products of local $L$-factors. These admit a simple description involving the
conjectural local Langlands correspondence; unconditional definition is known in some special cases only.

A local (non-archimedian) $L$-factor is a function of a complex variable $s$, which has the form $P(q^{-s})$, where $P$ is a polynomial whose free term is equal to 1. The (local) Langlands program predicts that
such an $L$-factor can be associated to a choice of

1) A local non-archimedian field $\sF$ with the ring of integers $\sO$.

2) A split connected reductive group $G$ over $\sF$.

3) A non-trivial character $\chi\colon G\to \GG_m$.

4) An algebraic representation $\rho$ of the dual group $G^{\vee}$ such that the corresponding co-character $\chi^\svee$
of the center of $G^{\vee}$ is positive with respect to $\rho$ (i.e. that the restriction of $\rho$ to $\GG_m$ given by
$\chi^\svee$ is a sum of characters of the form $z\mapsto z^i$ with $i>0$).

5) A smooth irreducible representation $\pi$ of $G(\sF)$.

We will denote it by
$L_{\chi}(\pi,\rho,s)$.
\subsection{$L$-functions of Galois representations}
Let $\WW$ denote the Weil-Deligne group of $\sF$ and let $\II$ be the corresponding inertia subgroup. This is a normal
subgroup of $\WW$ such that the quotient $\WW/\II$ is isomorphic to $\ZZ$. We shall denote the generator of this
quotient by $\Fr$ (the Frobenius element).

Let $\sig\colon \WW\to \GL(V)$ be a representation of $\WW$ (here $V$ is a finite-dimensional vector space over $\CC$). Then
the $L$-function of $\sig$ is defined to be
\begin{equation}
  \label{lsigma}
L(\sig,s)=\frac{1}{\det(1-q^{-s}\sig(\Fr)|_{V^{\II}})}.
\end{equation}

For the purposes of this paper we need some variant of the above definition.
Let $\pi,\rho,\chi$ and $G$ be as before.
The local Langlands conjecture (which is now proved in many cases)
is supposed to associate to $\pi$ a homomorphism from $\WW$ to $G^{\vee}$.
Composing it with $\rho$ we get a representation $\sig_{\pi,\rho}\colon \WW\to \GL(V)$. In this case the function $L_{\chi}(\pi,\rho,s)$
is supposed to have an interpretation similar to~\eqref{lsigma}. Namely, we should
have
\begin{equation}
  \label{lrho}
L_{\chi}(\pi,\rho,s)=\frac{1}{\det(1-\chi^\svee(q^{-s})\sig_{\pi,\rho}(\Fr)|_{V^{\II}})}.
\end{equation}
The question that remains is how to define the left-hand side without appealing to the Langlands conjectures.

\subsection{Godement-Jacquet construction}In the special case when $G=\GL(n)$, $\chi=\det$ and
$\rho=\CC^n$ is the tautological representation of $G^{\vee}=\GL(n,\CC)$ a remarkable definition of
$L_{\chi,\rho}(\pi,s)$ has been proposed by Godement and Jacquet~\cite{gj}.
Namely, let $\phi$ be a locally constant function on $\Mat(n,\sF)$ (the space of $n\x n$-matrices
over $\sF$) with compact support and let $c$ be a matrix coefficient of $\pi$.
Let us set
\begin{equation}\label{GJint}
Z(\phi,c,s)=\int_{G(\sF)} \phi(g)c(g)|\det(g)|^s dg.
\end{equation}

Godement and Jacquet show that this integral is absolutely convergent for $Re(s)\ll0$ and in fact defines a rational
function of $q^s$. Moreover, for a given $\pi$ all the functions $Z(\phi,c,s+\frac{n-1}{2})$ form a fractional ideal
of the ring $\CC[q^s,q^{-s}]$ and the function $L(\pi,s)$ is defined to be the generator of this ideal
(normalized by the condition that $L=1$ when $q^{-s}=0$).

The shift by $\frac{n-1}{2}$ is in fact a normalization issue which will automatically (almost) disappear in our general construction.

\subsection{The general case}
The Godement-Jacquet procedure is based on the space $\calS$ of locally constant compactly supported functions on $\Mat(n,\sF)$.
So, for general $G$ and $\rho$ one needs to find its analog $\CS_{\rho}$. It is supposed to be some space of locally constant
functions on $G$ containing all functions with compact support, which depends on $\rho$.

Assume that $G$ is the group of invertible elements of a normal algebraic semi-group $\oG$.
Using the work of Vinberg one can associate to $\oG$ a representation $\rho$ of $G^{\vee}$. It is tempting to define $\calS_{\rho}$ as the space of locally
constant compactly supported functions on $\oG$.
However, in general this idea turns out to be too naive as the semi-group
$\oG$ is usually singular.

Let us now assume that $\sF$ is $\FF_q(\!(t)\!)$.
The 
following idea for definition of $\CS_{\rho}$ was discussed in \cite{BNS}, partly building on \cite{bk1}, \cite{bk2}.
 Let $G_\sF$ denote  denote the
ind-scheme of maps from $\Spec\FF_q(\!(t)\!)$ to $G$ and let $\oG{}^0_\sF$ denote the ind-scheme of maps
from $\Spec\FF_q(\!(t)\!)$
to $\oG$ which land in $G$; note that on the level of $\FF_q$-points $G_\sF$ and $\oG{}^0_\sF$ are the
same, but as ind-schemes they are very different.
Then the idea is that  $\calS_{\rho}$ should be the space generated by the trace of Frobenius
functions associated to irreducible perverse sheaves on the space $\oG{}^0_\sF$.

This idea was realized in some cases in \cite{BNS};  however, there are substantial obstacles for realizing it in general.
First, not every representation $\rho$ can be associated to a semi-group
$\oG$. Second, it is technically challenging to define the above category of perverse sheaves.
Our method instead is based on a {\em global model} for the singularities of $\oG$ which can be defined
for an arbitrary $\rho$.

The precise connection with semi-groups is discussed in the appendix.

\subsection{The fusion construction of the $L$-factors} The global model above
 is defined starting from a global curve $C$ with a marked point $x_0$, as opposed to the formal disc featuring in local constructions.
More precisely, we consider the so
called ramified, symmetrized Beilinson-Drinfeld Grassmannian $\Gr_{n}$; or rather a closed subvariety  $\Gr_{n}^\rho\subset \Gr_{n}$ depending 
on the highest weights of $\rho$. The variety  $\Gr_{n}^\rho$ fibers over the configuration space $C^{(n)}$.

Taking as an input  the image of the  representation $\rho$ under the geometric Satake equivalence and an arbitrary sheaf $\CF$ on the loop group
we define a perverse sheaf $\XF_{\rho,(n)}$ on   an open 
part $\Gr_{x_0,n}^{\on{dis}}\subset \Gr_{n}^\rho$, the preimage of the space of configuration of $n$ distinct points in $C\setminus \{x_0\}$.
Applying minimal extension to such a sheaf and restricting it to the locus where each point in the configuration equals $x_0$ ,
we get an $\ell$-adic complex on a space receiving a map from the loop group $G_{\sF}$.
The space $\calS_\rho(G)$ is then defined as the span of Frobenius trace functions produced from those complexes.

We prove that the $L$-factors produced by the Godement-Jacquet construction from this space 
has expected properties in several cases.
If $G=\GL(n)$ and $\rho$ is the tautological representation we check that our definition coincides with the classical one \cite{gj}.
In the following cases where local Langlands correspondence is well established we check that our definition agrees with \eqref{lrho}: when 
$G$ is abelian, or when  $G$ is arbitrary but the sheaf $\CF$ is equivariant under
the group of positive loops, or more generally when it is equivariant with respect to an Iwahori subgroup.
Of these, the latter case is most involved,
we finish the introduction by describing the key ideas of the proof in that case.

\subsection{The Iwahori case}

The two main ingredients in the proof are: 1) the description of the minimal extension of a perverse sheaf in terms of Beilinson's gluing data \cite{bei} and its generalization 
presented in ~\S\ref{BD}, and 2) the equivalence between the two categorifications of the affine Hecke algebra \cite{BHecke} and its properties.

According to \cite{bei}, given an algebraic variety $X$ with a function $f$ and a perverse sheaf $\CF$ on $U=X\setminus (f)$ we have
$i^!j_{!*}(\CF)[-1]=\on{CoKer}(\Psi_f(\CF)\overset{1-m}{\longrightarrow} \Psi_f(\CF))$ where  $i$, $j$ are the embeddings of $D=(f)$ and $U=X\setminus D$ into $X$ respectively and $\Psi$
is the {\em nearby cycles} functor with its monodromy automorphism $m$. 

In particular, if $f$ is a (local) function on $\Gr_{1}^\rho$ coming from a coordinate function of $C$ defined in a neighborhood of $x_0$, then
 $\Psi_f(\XF_{\rho,(1)})$  yields  convolution of $\CF$ with the {\em central sheaf} $\sZ_\rho$ of \cite{g} (here we assume
 that $\XF$ is constructed from an Iwahori equivariant sheaf $\CF$). Combining this observation with the previous paragraph we identify the restriction 
 of the minimal extension of $\XF_{\rho,(1)}$ as coinvariants of monodromy acting on convolution $\CF*\sZ_\rho$ of $\CF$ with the central sheaf $\sZ_\rho$. The  multi-parametric generalization of \cite{bei} in   ~\S\ref{BD}  allows one to write the restriction of  $\XF_{\rho,(n)}$ as coinvariants of the monodromy automorphism\footnote{Here we are
 referring to  monodromy  acting on one tensor factor in the symmetric power of $\rho$. Notice that this is different from monodromy action coming from the nearby cycles
 construction applied to the representation $\on{Sym}^n(\rho)$.} of $\CF*\sZ_{\on{Sym}^n(\rho)}$.

Next we invoke  a description of the category of Iwahori equivariant
 perverse sheaves on the affine flag variety in terms of coherent sheaves on  the Steinberg variety of the dual group  \cite{BHecke}.  Notice that this equivalence intertwines 
 the central functor $\sZ_V$
 with the functor $\CF\mapsto V\otimes \CF$ of tensoring by a representation of $\LG$. These results allow one to compare the construction in the previous paragraph with
 pull back of a complex of coherent sheaves from $\St$ to $Z_\rho$, see~Theorem \ref{main-intro-exp}(2) for notation. This yields 
  a ``coherent" description of the corresponding subspace stated in Theorem \ref{main-intro-exp}(2) and Theorem \ref{Th1}, which formally implies~\eqref{lrho}.

\subsection{Organization of the paper} The paper is organized as follows. In~\S\ref{formulation}
we give a precise definition of the space $\calS_\rho(G)$ and formulate some conjectures about it.
We prove the conjecture in the standard example of the group $\GL(N)$ and its tautological
representation, and state the main~Theorem~\ref{main-intro} that asserts validity of the conjecture
for Iwahori-equivariant sheaves.
In~\S\ref{toric} we consider the case when $G$ is a torus. In~\S\ref{iwahori} we describe the plan
of our proof of the main~Theorems~\ref{main-intro},~\ref{main-intro-exp}.
This plan is worked out in~\S\ref{BD} (calculations
on the geometric side involving the ramified Beilinson-Drinfeld Grassmannian) and~\S\ref{proofs_RB}
(calculations on the spectral side involving an enhanced Steinberg variety). Finally, in
Appendix~\S\ref{semi-groups} we discuss the relations with arc spaces of reductive semigroups.

\subsection{Acknowledgments}
We are indebted to P.~Achar and S.~Riche for sharing~\cite{ar} with us prior to its publication.
We are especially grateful for their work on~\cite{ar2} supplying the proof
of~Proposition~\ref{nadler}.
We also thank A.~Bouthier, D.~Nadler, B.~C.~Ng\^o and Y.~Sakellaridis
for useful discussions.

We dedicate the article to the memory of Yuri Ivanovich Manin, whose influence on the subject and on the authors can not be overestimated.
We refer, in particular, to \cite{m} for an inspiring vision of the place of local $L$-factors in a rather different mathematical landscape.

R.B.~was supported by NSF grant No.~DMS-1601953.
A.B.~was partially supported by NSERC.
M.F.~was partially supported by the Basic Research Program at the HSE University.
D.K.~was partially supported by ERC grant 669655.

\section{The main conjectures: precise formulation}\label{formulation}

\subsection{The basic functions}\label{basic}
Let $\CF$ be any irreducible perverse sheaf on $G_\sF$ (this is by definition an irreducible
perverse sheaf on the ind-scheme $G_\sF/K_i$ for some congruence subgroup $K_i$).
To $\CF$ we are going to associate a function
$f_{\CF,\rho}$ on $G(\sF)$. More precisely, 
we are going
to define a sequence of complexes of sheaves $\XF_{\rho,(n)},\ n\geq 0$, on $G_\sF$ such that
\begin{equation}\label{XF_def}
f_{\CF,\rho}=\sum_n f_{\XF_{\rho,(n)}},
\end{equation}
where $f_{\XF_{\rho,(n)}}$ is the function corresponding to $\XF_{\rho,(n)}$ by the
``sheaf-function" correspondence,
 i.e.\ $f_{\XF_{\rho,(n)}}(x)$ is defined to be the trace of
 the Frobenius acting on the corresponding stalk of $\XF_{\rho,(n)}$.


Let $C$ be a smooth projective curve over $\FF_q$ with an $\FF_q$-point $x_0\in C$ and with a local parameter $t$ near $x_0$.
Let $\Gr_{x_0,n}$ denote the moduli space of the following data:

1) An {\em unordered} $n$-tuple of points $(x_1,\ldots,x_n)$ of $C$;

2) A $G$-bundle $\CP$ on $C$;

3) A trivialization of $\CP$ away from $(x_0,x_1,\ldots,x_n)$;

4) A trivialization of $\CP$ in the formal neighbourhood of $x_0$.

\noindent
We have a natural projection $p_n\colon \Gr_{x_0,n}\to G_\sF$ and $\pi_n\colon  \Gr_{x_0,n}\to \Sym^n C$.
In addition $G_\sF$ sits in $\Gr_{x_0,n}$ as a closed subset corresponding to the locus when all $x_i$
are equal to $x_0$; we denote the corresponding embedding by $i_n$.
Let also $j_n$ denote the embedding of the open subset $\Gr_{x_0,n}^{\on{dis}}$ into $\Gr_{x_0,n}$ where
$\Gr_{x_0,n}^{\on{dis}}$ is the locus where all $x_i$ are distinct.

Note that every fiber of the map $\pi_n$ restricted to $\Gr_{x_0,n}^{\on{dis}}$ is isomorphic to
$G_\sF\times \Gr_G^n$ (here $\Gr_G$ denotes the affine Grassmannian of $G$). This isomorphism is not
completely canonical, but it is canonical up to the action of the group of automorphisms of
$\Spec\FF_q[\![t]\!]$ acting on every factor in $\Gr_G$ and also up to the action of the symmetric
group $S_n$ permuting those factors. Thus every perverse sheaf on $G_\sF\times\Gr_G^n$ which is
equivariant with respect to these actions gives rise to a perverse sheaf on $\Gr_{x_0,n}^{\on{dis}}$.
Let $\CF$ be an irreducible perverse sheaf on $G_\sF$. Let also $V$ be the vector space of
representation $\rho$ of $G^\vee$, and let $\sfS(V)$ be the (semi-simple) perverse sheaf on $\Gr_G$
corresponding to $V$ under the geometric Satake equivalence. Then the sheaf
$\CF\boxtimes\sfS(V)^{\boxtimes n}$ is a perverse sheaf on $G_\sF\times \Gr_G^n$ with the above
equivariance properties. We denote by $\CF_{\rho,(n)}^\dagger$ the corresponding perverse sheaf
on $\Gr_{x_0,n}^{\on{dis}}$.

We now set
\begin{equation}\label{XF_n}
\XF_{\rho,(n)}=i_n^*(j_{\rho,n})_{!*}\CF_{\rho,(n)}^\dagger.
\end{equation}

\subsection{The space}
We define the space $\calS_{\rho}(G(\sF))$ as the space generated by all the functions $f_{\CF,\rho}$
defined above. It is clear that $\calS_{\rho}(G(\sF)))$ is invariant under left and right action of
$G(\sF)$. Also, every element of $\calS_{\rho}(G(\sF))$ is a locally constant function on $G(\sF)$.


\subsection{The standard case}
\label{the standard case}
Let $G=\GL(N)$ and let $(V,\rho)$ be the standard representation of $G^\vee\simeq\GL(N)$.
Then we claim that in this case
$\calS_{\rho}(G(\sF))$ consists of functions of the form $\phi(g)|\det(g)^{\frac{N-1}{2}}|(-1)^{v(\det(g))}$
where $\phi$ is a locally constant compactly supported function
on the space $\Mat(N,\sF)$ of $N\x N$-matrices over $\sF$ (here  $v\colon \sF^{\times}\to \ZZ$ is the valuation map). For this it enough to prove the following two statements: 


(i) For any $\phi\in \calS_{\rho}(G(\sF))$ the function  $\phi(g)|\det(g)^{-\frac{N-1}{2}}|(-1)^{v(\det(g))}$
extends to a locally constant compactly supported function on $\Mat(N,\sF)$.

(ii) For any $A\in \Mat(N,\sF)$ and for all $k\geq 0$ the constant function on $A+t^k\Mat(N,\sO)$ divided by
$|\det(g)^{\frac{N-1}{2}}|(-1)^{v(\det(g))}$ belongs to $\calS_{\rho}(G(\sF))$.

Statements (i) and (ii) follow from the following ``local" description of the sheaves $\CF_{\rho,(n)}$.
Let $\CF$ be as in~\S\ref{basic}.

Then $\CF$ is equivariant under a congruence subgroup $K_i$ of
$G_{\sF}$. It follows that for $j$ sufficiently large the following things are true: 

1) The closure of the support of $\CF$ in $ \Mat(N,\sF)
$ is invariant under shifts by $t^j\Mat(N,\sO)$;

2) The image  supp$(\CF)$ in  $\Mat(N,\sF)/t^j\Mat(N,\sO)$ is a locally closed subset $Z$;

3) There exists a (perverse up to a shift) sheaf $\overline\CF$ on $Z$ such that $\CF$ (viewed as a sheaf on its support) is equal to the pull-back of $\overline{\CF}$.

Consider now the Goresky-MacPherson extension of
$\overline\CF$ to $\Mat(N,\sF)/t^j\Mat(N,\sO)$ and pull it back to $\Mat(N,\sF)$.
We denote the resulting sheaf by $\CG$, and the corresponding trace of Frobenius function is
denoted by $\phi$. The desired description of $\calS_\rho(G(\sF))$ follows from the next

\begin{lem}\label{gln-lem} Assume that $v(det)=d$ on the support of $\CF$. Then
  we have $f_{\CF,\rho}(g)=\phi(g)|\det(g)^{\frac{N-1}{2}}|q^{-\frac{d(N-1)}{2}}(-1)^{v(\det(g)-d)}$.
\end{lem}

\proof
Since the construction is invariant under the action of $\sF^\times$ on $\Mat(N,\sF)$ by dilations,
we can assume without loss of generality that $d=0$.
Then $\calG$ is supported on  $\Mat(N,\sO)$. Let $\CG_n$ be the
$*$-restriction of $\CG$ to the locus  $v(\det)=n$, a
connected component of $G_{\sF}$.

Then we claim that
\begin{equation}\label{gln-sheaf}
\XF_{\rho,(n)}=\CG_n[n(N-1)]\left(\frac{n(N-1)}{2}\right),
\end{equation}
and (\ref{gln-sheaf}) clearly implies that
\begin{equation}\label{gln-function}
f_{\CF,\rho}(g)=\phi(g)|\det(g)^{\frac{N-1}{2}}|  
(-1)^{v(\det(g))}    ,
\end{equation}
which is what we need to prove. Thus it remains to show the validity of (\ref{gln-sheaf}).
For this, it is enough to show the following.

Let $\Gr_{x_0,n}^\rho$ denote the closure of the subset of
$\Gr_{x_0,n}^{\on{dis}}$  and whose fiber over an $(n+1)$-tuple of distinct points
$(x_0,\ldots,x_n)$ is equal to $\Mat(N,\sO)\times (\Gr_G^\rho)^n$ where $\rho$ is again the
standard representation of $\GL(N)$ and $\Gr_G^\rho$ is the support of the sheaf corresponding to
$\rho$ under the geometric Satake equivalence (this is a closed subvariety of $\Gr_G$ isomorphic to
$\PP^{N-1}$). Note that we have a natural map $\Gr_{x_0,n}^\rho\to \Mat(N,\sO)$ (defined by restricting all data to the formal neighbourhood of $x_0$). We need to show that this map is formally smooth
(the shift and the Tate twist in (\ref{gln-sheaf}) come from the fact that  dimension of the support
of $\sfS(V)$ is equal to $N-1$ in this case).

It is easy  
to see that $\Gr_{x_0,n}^\rho$ is equal to the moduli space of the following data:

1) A vector bundle $E$ on $C$ of rank $N$ and degree $n$.

2) An injective morphism of sheaves $\phi:\calO_C^N\to E$;

3) A trivialization of $E$ in the formal neighbourhood of $x_0$.

\noindent
The restriction of $\phi$ to the formal neighbourhood of $x_0$ gives a $\sO$-linear map
$\sO^N\to\sO^N$, i.e.\ an element of $\Mat(N,\sO)$. It is enough to check the formal smoothness
after taking quotient by $\GL(N,\sO)$. To do this it is enough to show the following:

a) The quotient $\Gr_{x_0,n}^\rho/\GL(N,\sO)$ is a smooth scheme of finite type over $\FF_q$.

b) The map $\Gr_{x_0,n}^\rho/\GL(N,\sO)\to \Mat(N,\sO)/\GL(N,\sO)$ induces a surjection on
tangent spaces.


\noindent
Note that $\Gr_{x_0,n}^\rho/\GL(N,\sO)$ is just the moduli space of data 1) and 2) as above. Thus the fact that it is a scheme of finite type over $\FF_q$ is standard. Also, the tangent space to this scheme is equal to the following.
We have the natural injective map $\underline{\End}(E)\to \underline{\Hom}(\calO_C^N,E)$ given by composition with $\phi$.
Let $\calK(E,\phi)$ denote the cokernel of this map; this is a torsion sheaf on $C$ of length $Nn$. Then the sought-for tangent space is just $H^0(C,\calK(E,\phi))$. In particular, its dimension is always equal to $Nn$, hence we see that
$\Gr_{x_0,n}^\rho/\GL(N,\sO)$ is a smooth scheme of dimension $Nn$.
 
It remains to check the statement b) above. For this let us note that the tangent space to $\Mat(N,\sO)$ at a matrix $A$ is equal to $\Mat(N,\sO)/A\cdot \Mat(N,\sO)$ (note that this only depends on the orbit of $A$ under right multiplication by $\GL(N,\sO)$). Assume that $A$ comes from $(E,\phi)$ as above. Then $\Mat(N,\sO)/A\cdot \Mat(N,\sO)$ is precisely the space of global sections of the direct summand of the sheaf $\calK(E,\phi)$ concentrated at $x_0$. This shows that the map $H^0(C,\calK(E,\phi))\to \Mat(N,\sO)/A\cdot \Mat(N,\sO)$ is surjective.
\qed

Lemma \ref{gln-lem} immediately implies statement (i) above. It also implies statement (ii), since $A+t^k\Mat(N,\sO)$ is a smooth closed subset of $\Mat(N,\sF)$.

\subsection{The twist}

To formulate the precise statement we need one more ingredient.
Let $2\del$ be the sum of the positive roots of $G$.
Let $\Lambda$ denote the coweight lattice of $G$ and let also $Q$ denote the coroot lattice of $G$.
Recall that we have a canonical homomorphism $G(\sF)/[G(\sF),G(\sF)]\to \Lambda/Q$.


Thus the element $2\del$ is a homomorphism from $\Lambda$ to $\ZZ$
which sends $Q$ to $2\ZZ$. Hence it can be regarded as a homomorphism $\Lambda/Q\to \pm 1$ and hence as a character
$\sgn\colon G(\sF)\to \pm 1$.
For a representation $\pi$ of $G(\sF)$ we set $\pi'=\pi\otimes \sgn$.



\subsection{The main conjecture}
We now turn to our main conjecture.
In what follows $dg$ will denote a Haar measure on $G(\sF)$.

\begin{conj}
    \label{main-conj}
\begin{enumerate}
\item $\calH(G(\sF))\subset \calS_{\rho}(G(\sF))$.
  \item
  Let $c(g)$ be a matrix coefficient of an irreducible representation $\pi$ of $G(\sF)$. Then the
  integral
  \begin{equation}
    \label{integral}
  \int_G \phi(g)c(g)|\chi(g)|^s dg
  \end{equation}
  for $\phi\in \calS_{\rho}(G(\sF))$
  is convergent for $Re(s)\gg0$ and it is a rational
  function of $q^s$.
  \item
  For a given $\pi$ as above the set of all rational functions given by \eqref{integral} for all matrix coefficients of
  $\pi$  and all test functions $\phi\in \calS_{\rho}(G(\sF))$ is a fractional ideal of $\CC[q^s,q^{-s}]$, whose generator $L(\pi',\rho,s)$ is the local $L$-function of $\pi'$
  associated with the representation $\rho$.

  \item
  Let $K$ be an open compact subgroup of $G_\sF$ assigned to a compact subset in the building by Bruhat-Tits theory 
  (for example, $K$ can be a congruence subgroup or a parahoric subgroup). Abusing slightly the notation we shall denote the corresponding pro-algebraic subgroup of $G_\sF$ by the same symbol. Then $\calS_{\rho}(G(\sF))^{K\times K}$ is generated by functions
  $f_{\CF,\rho}$, where $\CF$ runs through $K\times K$-equivariant irreducible perverse sheaves on $G_\sF$.
\end{enumerate}
\end{conj}

\noindent
\begin{rem}
  For $G=\GL(N)$ and $\rho$ the tautological representation, Conjecture~\ref{main-conj}(1-3) follows
  from~\cite{gj} and~\S\ref{the standard case}. Assertion 4 also follows easily in the following way: assume that
  $\calF$ is an irreducible perverse sheaf on $G_\sF$. Let $\calG$ be as above. Then  $\calG$ is an irreducible perverse sheaf on $\Mat(N,\sF)$. Assume that the corresponding function is $K\times K$-invariant, but $\calF$ (and thus $\calG$) is equivariant with respect to some smaller congruence subgroup $\Gam$. Let $\calG_K$ be the $!$-averaging of $\calG$ with respect to $K/\Gam\times K/\Gam$. Then up to a constant its trace function does not change.
  However, $\calG_K$ is an object of the bounded derived category of $K\times K$-equivariant sheaves on $\Mat(N,\sF)$, and thus every irreducible subquotient $\calG_{\alp,i}$ of every perverse cohomology $^pH^i(\calG_K)$ of $\calG_K$ is $K\times K$-equivariant (here $\alp$ runs over some set of indices, which depends on $i$; the fact that only finitely many $^pH^i(\calG_K)\neq 0$ and that for all $i$ the sheaf $^pH^i(\calG_K)$ is of finite length follows from the fact that in this case averaging over $K/\Gamma$ can be described by reducing to an averaging of a sheaf on a scheme of finite type -- in the way similar to the procedure discussed before Lemma \ref{gln-lem}). Each of those irreducible subquotients is equal to the Goresky-MacPherson extension of some irreducible perverse sheaf $\calF_{\alp,i}$ on $G_\sF$. Then the function $f_{\calF,\rho}$ obviously lies in the span of the functions $f_{\calF_{\alp,i}}$.
\end{rem}

\noindent
\begin{rem} Let $K$ be an open compact subgroup of $G(\sF)$  such that
$\pi^K\neq 0$. Then it is easy to see that in order to get the correct fractional ideal in (3), it is enough
to consider the integrals \eqref{integral} with both $\phi$ and $c$ left and right $K$-invariant.
\end{rem}

On the other hand, assume that $K$ is the group of $\FF_q$-points of an algebraic group subscheme of $G_\sF$ containing a sufficiently small congruence subgroup (as before we shall denote this group scheme by the same letter $K$).
Let us denote by  $\calS_{\rho}(G,K)$ the space generated by functions
  $f_{\CF,\rho}$, where $\CF$ runs through $K\times K$-equivariant irreducible perverse sheaves on $G_\sF$. Clearly we have $\calS_{\rho}(G,K)\subset \calS_{\rho}(G(\sF))^{K\x K}$. Note that $\calS_{\rho}(G(\sF))^{K\x K}$ is obviously a bimodule
over the Hecke algebra $H(G,K)$, while $\calS_{\rho}(G,K)$ is a priori not (although according
to~Conjecture~\ref{main-conj}(4) these two spaces coincide, at least when $K$ is nice enough).
Given $K$ as above we shall say that Conjecture~\ref{main-conj} holds for $\calS_{\rho}(G,K)$ if the
statements (1--4) of Conjecture~\ref{main-conj} hold when $\calS_\rho(G(\sF))$ is replaced by
$\calS_\rho(G,K)$; $\calH(G(\sF))$ is replaced by $H(G,K)$ (the Hecke algebra of $G(\sF)$
corresponding to $K$); representation $\pi$ is such that $\pi^K\neq 0$; and all matrix coefficients
appearing in (\ref{integral}) are assumed to be $K\times K$-invariant.

The main result of this paper is the following
\begin{thm}
  \label{main-intro}
  Let $K$ be either $G(\sO)$ or the Iwahori subgroup $I$. Then the statements (1)-(3)
  of~Conjecture~\ref{main-conj} hold
 for $\calS_{\rho}(G,K)$ (instead of the full $\calS_{\rho}(G(\sF))$).
\end{thm}

Note that if we assume~Conjecture~\ref{main-conj}(4) then~Theorem~\ref{main-intro} says that the
space $\calS_\rho(G(\sF))^{K\times K}$ computes the correct $L$-function for representations $\pi$ which
have a non-zero $K$-fixed vector (note also that we do know the definition of $L(\pi,\rho,s)$ for such
representations, since the local Langlands conjecture is known for representations with Iwahori-fixed
vector).

More precisely, we are going to prove the following:


\begin{thm}
  \label{main-intro-exp}
\begin{enumerate}
\item
  The space $\calS_{\rho}(G,G(\sO))$ is equal to $f_{\rho}\star H(G,G(\sO))$ where $f_{\rho}$ is the
  function on $G(\sF)$ corresponding to the representation $\Sym(V)$ under the Satake isomorphism.
  Here $V$ is the vector space of representation $\rho$.

\item Let $\St$ be the Steinberg variety of $G^{\vee}$ classifying triples $(x,\grb_1,\grb_2)$ where each $\grb_i$ is a Borel subalgebra of $\grg^\svee=\on{Lie}(G^{\vee})$ and $x$ is a nilpotent element in $\grb_1\cap \grb_2$. Consider the equivariant $K$-theory $K^{G^{\vee}\times \GG_m}(\St)$. This is an algebra over $\CC[\bfq,\bfq^{-1}]=K^{\GG_m}(\pt)$. Identify $H(G,I)$ with $K^{G^{\vee}\times \GG_m}(\St)_q$ where the latter stands for the specialization of $K^{G^{\vee}\times \GG_m}(\St)$ to $\bfq=q$ \cite[\S 7.2]{CG}.
 Then $\calS_{\rho}(G,I)$ is naturally isomorphic to $K^{G^{\vee}\times \GG_m}(Z_{\rho})_q$ where $Z_{\rho}$ is the scheme classifying quadruples $(x,\grb_1,\grb_2,\xi)$ where $(x,\grb_1,\grb_2)$ are as before and $\xi\in V^*$ is such that
    $x(\xi)=0$.
\end{enumerate}
\end{thm}

See Theorem \ref{Th1} for a more precise version of Theorem \ref{main-intro-exp}(2).

Theorem~\ref{main-intro} will be deduced from~Theorem~\ref{Th1} in~\S\ref{deduction}.

\subsection{Changing $\rho$}
We conclude this Section by adding another conjecture to Conjecture~\ref{main-conj}. 
Notice that the formula \eqref{lrho} makes it clear that the local $L$-factor defined by means
of local Langlands correspondence satisfies $L(\pi, \rho_1\oplus \rho_2; s) =L(\pi, \rho_1; s) \cdot L(\pi, \rho_2; s)$. 
This motivates the following 


\begin{conj}\label{conv}
Let $\rho=\rho_1\oplus\rho_2$. Then $\calS_{\rho}(G)$ is the span of functions of the form $\phi_1\star \phi_2$ where
$\phi_i\in \calS_{\rho_i}(G)$.
\end{conj}

\section{Toric case}\label{toric}
\subsection{The statement} In this Section we consider $G=T$: a split torus. We would like to give another definition of the space $\calS_\rho(T(\sF))$ and prove that it is equivalent to the original one. This will allow us to prove Conjecture~\ref{main-conj} in the toric case.

The representation $\rho$ of $T^{\vee}$ is essentially given by an $n$-tuple of weights $(\lam_1,\ldots,\lam_n)$ of $T^{\vee}$ which can be regarded as coweights of $T$. Thus we get a natural morphism
$p_\rho\colon\GG_m^n\to T$ sending $(x_1,\ldots,x_n)$ to
$\prod \lam_i(x_i)$.

We shall say that {\em $\rho$ is non-degenerate} if the map $p_{\rho}$ is surjective.
If $\rho$ is degenerate, we let $T'$ denote the image of $p_{\rho}$. Then $\rho$ can be regarded as a non-degenerate
representation of $(T')^{\vee}$. Set also $T''=T/T'$.

Let us assume that $\rho$ is non-degenerate. Choose $\phi\in \calS(\sF^n)$ (that is a locally constant
compactly supported function on $\sF^n$). Let us denote by $(p_\rho)_!(\phi)$ the locally constant
function on $T(\sF)$ such that $(p_\rho)_!(\phi)dt$ is equal to the direct image of the distribution
$\phi |d^*x_1\cdots d^*x_n|$ (restricted to $(\sF^{\times})^n$ under the map $p_\rho$). It is easy to see that $(p_\rho)_!(\phi)$ is well defined and that it is locally costant on $T(\sF)$.

\begin{prop}
Assume Conjecture \ref{conv}. Then the following are true:
\begin{enumerate}
\item
Assume that $\rho$ is non-degenerate. Then the space $\calS_{\rho}(T)$ consists of functions of the form $(p_\rho)_!(\phi)$.
\item
For general $\rho$ the space $\calS_{\rho}(T)$ consists of locally constant functions $\phi$ on $T(\sF)$ such that

\textup{a)} The image of the support of $\phi$ in $T''(\sF)$ is compact.

\textup{b)} The restriction of $\phi$ to any fiber of the map $T(\sF)\to T''(\sF)$
(which can be canonically up to shift identified with $T'(\sF)$) lies in $\calS_{\rho}(T')$
(for which we have an explicit description by~\textup{(1)} above).

\noindent
For example, when $\rho=0$ we get $\calS_{\rho}(T)=\calS(T)$.

\end{enumerate}
\end{prop}

\begin{proof}
Conjecture \ref{conv} immediately reduces the statement to the case $n=1$. This case by itself immediately reduces to the case $T=\GG_m$, $\rho=k$-th power of the standard character of $T^{\vee}=\GG_m$.
\end{proof}

\section{Proof of Theorems \ref{main-intro}, \ref{main-intro-exp}}\label{iwahori}

\subsection{The spherical case}
\label{spherical case}
In this subsection we prove~Theorem~\ref{main-intro-exp}(1) and deduce from it
the spherical case of~Theorem~\ref{main-intro}. This is in fact immediate:
namely, it follows easily that for a representation $(V,\rho)$ of $G^\vee$ and for a
(left and right) $G(\sO)$-equivariant perverse sheaf $\CF$ on $G(\sF)$, the sheaf $\XF_{\rho,(n)}$
is equal to $\CF\star\sfS(\Sym^n(V))$ (where the convolution
is taken in the category of $G(\sO)$-equivariant sheaves). Hence
\begin{equation}\label{spherical}
f_{\CF,\rho}=f_\CF\star f_{\sfS(\Sym(V))}.
\end{equation}

To complete the proof of the spherical case in~Theorem~\ref{main-intro} we 
make a small digression on the normalization of
the Satake isomorphism. Recall (cf.~e.g.~\cite{Gross}) that we have a canonical isomorphism
$S\colon \sK_0(\Rep(G^{\vee}))\iso H(G(\sF),G(\sO)$ where $\sK_0(\Rep(G^{\vee}))$
stands for the complexified Grothendieck group of $\Rep(G^{\vee})$. Let $\eps$ be the automorphism of
$K_0(\Rep(G^{\vee}))$ which sends the class of a representation with highest weigh $\lambda\in \Lambda$
to itself multiplied
by $(-1)^{\langle \lambda,2\delta\rangle}$ (where $2\delta$ is the sum of positive roots of $G$, so that
$\eps$ is trivial if $G$ is simply connected).
Then it is easy to see that for any representation $(V,\rho)$ of $G^{\vee}$ we have
\begin{equation}\label{satake}
f_{\sfS(V)}= S(\eps([V])).
\end{equation}
It is clear now that (\ref{spherical}) and (\ref{satake}) together imply Theorem \ref{main-intro} for $K=G(\sO)$.

\subsection{Proof of Theorem \ref{main-intro}}\label{deduction}
We now deduce~Theorem~\ref{main-intro} in the Iwahori case $K=I$
from~Theorem~\ref{main-intro-exp}(2), or rather its more precise version Theorem \ref{Th1} below.

We denote by $H=H(G,I)=\Ce[I\bs G(\sF)/I]$ the Iwahori-Hecke algebra of $G$.

Let $\phi\in H^*$ be a matrix coefficient of an irreducible $H$-module $M$ with parameters $(s,e,\psi)$.

Recall that by~\cite{kl,CG} $M$ is realized as a quotient of the group $[K(\CB_e^s):\psi]$,
the isotypic component for the action of the components group of the centralizer of $e$.
Here $\CB$ is the flag variety of $G^\vee$, and $\CB_e^s$ is its fixed point subvariety with
respect to the action of $s$ and $e$.

It follows that the linear functional $\phi$ factors as a composition of
$i_e^*\colon K^{G^\vee\times \Ce^\times}(\St)\to K^{Z_e}(\CB_e^2)$ and a linear functional
$\bar{\phi}\in K^{Z_e}(\CB_e^2)^*$. It also factors through the quotient by the maximal ideal of the
center of $H$ corresponding to the semi-simple element $(s,q)\in G^\vee\times\BG_m$.

Consider now the expression \eqref{GJint}, in view of the isomorphism of Theorem \ref{main-intro-exp}(2), the element $c\in \calS_\rho^I$ there can be
assumed without loss of generality to correspond to the class $[\CE]$ for $\CE \in \on{Coh}^{Z_e}(\CB_e^2\times V^e)$. 

Let $\on{pr}$ denote
the natural projection $\CB^2_e\times V^e\to\CB^2_e$.
The direct image $\on{pr}_*(\CE)$ can be decomposed as a direct sum $\bigoplus_n \CE_n$, 
where the cocharacter $\chi^\svee$ of $G^\vee$ acts on $\CE_n$ by the character $t\mapsto t^n$
(notice that the corresponding action of $\Gm$ on $\St$ is trivial).

Using the more precise description of the isomorphism in Theorem \ref{Th1}, we conclude that
$c=\sum\limits_n c_n$ where $c_n\in  H(G,I)\cong K^{G^{\vee}\times \GG_m}(\St)_q$ (see Theorem \ref{main-intro-exp}(2))
 corresponds to the class $[\CE_n]$. It is clear also that
 $c_n$ is supported on $G(\sF)_n=\{g\in G(\sF)\ |\ val(\chi(g))=n\}$.

 Thus the right hand side of  \eqref{GJint} takes the form:
 \[\sum_n q^{ns} \bar{\phi}(c_n).\]

The group $K^{Z_e}(\CB_e^2)$ is acted upon by the ring $K(\Rep(Z_e))$, so $K^{Z_e}(\CB_e^2)(\!(t)\!)$
is a module over $K(\Rep(Z_e))(\!(t)\!)$.
The standard argument involving the Koszul complex shows that the element
$\sum t^n c_n\in K^{Z_e}(\CB_e^2)(\!(t)\!)$ lies in $(1-t[V^e])^{-1} K^{Z_e}(\CB_e^2)[t,t^{-1}]$.
Hence the Laurent series $f_\CE(t)=\sum t^n  \bar{\phi}(c_n)$ lies in $P(t)^{-1}\Ce[t,t^{-1}]$
where $P(t)=1-\on{Tr}(q^s,V^e)$.
It remains to check that for some $\CE$ the Laurent series $f_\CE(t)$ generates the fractional ideal
$P(t)^{-1}\Ce[t,t^{-1}]$.
Let $\iota\colon\CB_e^2\to\CB_e^2\times V^e$ be the zero section embedding. We can find $\CE$ such
$\iota^*(\CE)=\iota^*(\CE)_0$ (i.e.\ the cocharacter $\chi^\svee$ of $G^\vee$ acts on
$\iota^*(\CE)$ trivially) and  $\langle \bar{\phi},
[\iota^*\CE]\rangle \ne 1$. Then $P(t)f_\CE(t)=1$. The claim follows.

\subsection{Setup}
The remainder of this section is devoted to proving~Theorem~\ref{main-intro-exp}(2).
So we will deal with the case $K=I$: the Iwahori subgroup.
In what follows we denote by $H=\Ce[I\bs G(\sF)/I]$ the Iwahori-Hecke algebra of $G$ and
we shall write $H_{\rho}$ instead of $\calS_{\rho}(G,I)$.
The main goal of this Section is to give a ``spectral" definition of $H_{\rho}$ and to prove that it coincides with the one given in Section \ref{basic}.

\subsection{Spectral side}
In order to introduce the spectral version of $H_\rho$ we recall that
\begin{equation}\label{Hisom}
\sH\otimes_{\Zet[\bq,\bq^{-1}]} \Ce \cong H,
\end{equation} where
$\sH$ is the affine Hecke algebra and the homomorphism $\Zet[\bq,\bq^{-1}] \to \Ce$ sends $\bq$ to $q^{-1}$.
Notice that a more standard choice a such an isomorphism sends $\bq$ to $q$. The isomorphism \eqref{Hisom}
coincides with that  isomorphism pre-composed with the Kazhdan-Lusztig involution of $\sH$; see Remark \ref{rem_costalk}
for a justification of that choice.


Furthermore,
$\sH\cong K(\Coh^{\LG\times \Gm}(\on{St}))$ where $\on{St}=\Nt\times_{\Lg} \Nt$ is the Steinberg
variety\footnote{Notice that several related constructions require considering $\on{St}$ as a
derived scheme but this does not affect the Grothendieck group.
In particular, we have $\sH=K(\Coh^{\LG\times \Gm}(\St))$ where  $\St=\gt\times_{\fg^\vee}\Nt$
is another scheme with the same reduced variety and the same Grothendieck group of coherent
sheaves. Since $\on{Tor}_{>0}^{\O(\Lg)}(\O(\gt), \O(\Nt))=0$ by a complete intersection argument,
the distinction between the derived and the ordinary scheme is immaterial in the case of $\St$.}
and $\bq$ is the standard generator of $K(\Rep(\Gm))$.
Note that isomorphism \eqref{Hisom} can be described as follows. One has a natural
map from $K(\Coh^{\LG\times \Gm}(\on{St}))$ to the Grothendieck group of Weil
sheaves on the affine flag variety (see \cite{BHecke}, especially \S 11.1),
then \eqref{Hisom} equals composition of that isomorphism with the map taking
the class of a Weil sheaf $\CF$ to the function of the trace of the {\em arithmetic} Frobenius
acting on the {\em costalks} of $\CF$; while the standard isomorphism sending
$\bq$ to $q$ corresponds to taking traces of geometric Frobenius on stalks of
the Weil sheaf.

Observe also that $H$ is $\BZ$-graded by the cocharacter $\chi^\svee\colon H=\bigoplus_{d\in\BZ}H_d$,
and we consider the corresponding completion $\widehat{H}:=\prod_{d\in\BZ}H_d$.

Let $Z_\rho\subset \St\times V^*$  be the closed subscheme parametrizing the quadruples
$(x,\fb_1,\fb_2,\xi)\in \Lg \times \CB\times\CB \times V^*$ such that
$x(\xi)=0$ (here $\CB$ is the flag variety of $G^\vee$). Then $Z_\rho$ is acted upon by
$\Gm\colon s\cdot(x,\fb_1,\fb_2,\xi)=(sx,\fb_1,\fb_2,\xi)$. We let
$\sH_\rho^\spec=K^{G^\vee\times \Gm}(Z_\rho)$ and
 $H_\rho^\spec=\sH_\rho^\spec/(\bq - q)$.

\subsection{Geometric side}
We now sketch the geometric description of the space $H_\rho$ outlined in~\S\ref{basic}
specifically in the Iwahori case. Since the representation $(V,\rho)$ of $G^\vee$ is fixed
throughout this Section, we will remove $\rho$ from the subscript of $\XF_{\rho,(n)}$.

Let $\sfS\colon \Rep(\LG)\to \Perv_{G_O}(\Gr_G)$ be the geometric Satake equivalence and
$\sZ\colon \Rep(\LG)\to \Perv_I(\Fl)$ be its composition with Gaitsgory's central functor~\cite{g}.

Let $\F\in \Perv_I(\Fl)$ be an  $I$-equivariant perverse sheaf on $\Fl$. We define a collection
$\bF =(\F_{(n)})$, $\F_{(n)}\in D_I(\Fl)$ as follows.
For every $n\geq 0$ consider the ramified version of Beilinson-Drinfeld Grassmannian $\Gr_n$ on a global
curve $C$ with a marked point $x_0\in C$. Thus $\Gr_n$
maps to $C^n$ where the fiber over $(x_1,\dots, x_n)$ is identified with $\Fl\times \Gr_G^n$ if
$(x_0,x_1,\dots,x_n)$ are pairwise distinct, while the fiber over
$(x_0,\dots, x_0)$ is identified with $\Fl$. We set:
\begin{equation}\label{Fndef}
\CF_{(n)}=(\imath^!\jmath_{!*}(\CF^\dagger_n))^{S_n}[n],
\end{equation} where
$\jmath$ is the embedding of the preimage of the complement to diagonals
in $(C\setminus \{ x_0\})^n$,
 $\F^\dagger_n$ is the sheaf on that preimage whose restriction to a generic fiber is identified
with $\F \boxtimes \sfS(V)\boxtimes \sfS(V)\boxtimes \cdots \boxtimes \sfS(V)$ and $\imath$
is the embedding of the preimage of $(x_0)^n$.

For $w\in W_\aff$, the irreducible $I$-equivariant perverse sheaf $\F=\IC_w$ carries a unique up to
scaling Weil structure. Fixing one
and using  a standard Weil structure (of weight zero)
on $\sfS(V)$ we get a Weil structure on $\F_{(n)}$.
Let $f_n^w$ be the function whose value at a point $x\in\Fl$ is the
 trace of the {\em arithmetic} Frobenius on the {\em costalk} of $\F_{(n)}$ at $x$.
 We set
$f^w = \suml_{n=0}^\infty f_n^w$ (notice that for every $d$ we have $f_n^w|_{G(\sF)_d}=0$
for almost all $n$, so $f^w\in \widehat{H}$ is well defined. Here $G(\sF)_d$ stands for the preimage
under $\chi$ of elements of $\BG_m(\sF)=\sF^\times$ of valuation $d$).

Notice that $(\F^\dagger_n)^{S_n}$ is Verdier self-dual, thus $\CF_{(n)}$ is  Verdier dual to
the sheaf $\XF_n$ defined in \eqref{XF_n}; thus  $f^w=f_{\IC_w,\rho}$ where the RHS
was defined in~\eqref{XF_def} (the need to introduce this Verdier dual description of the
function is explained in~Remark~\ref{rem_costalk}).

Let $H_\rho$ be the subspace in $\widehat{H}$ 
spanned by $f^w$, $w\in W_\aff$.

\subsection{Comparison}
Theorem \ref{main-intro-exp}(2) amounts to an isomorphism 
 of $H$-bimodules
 \begin{equation}\label{geom_to_spec_isom}
 H_\rho^\spec\cong H_\rho
\end{equation}
A  more precise version of this statement is as follows.

For an object $\CE\in \Coh^{G^\vee\times \Gm}(Z_\rho)$ consider $\pr_*(\CE) \in \QCoh^{G^\vee\times \Gm}(\St)$
where $\pr\colon Z_\rho\to \St$ is the projection.
The action of the multiplicative group $\chi^\svee(\Gm)\subset G^\vee$ on $\St$ is trivial, so the
object $\pr_*(\CE)\in  \QCoh^{G^\vee\times \Gm}(\St)$ splits canonically as
$\pr_*(\CE)=\oplusl\pr_*(\CE)_n$ where the multiplicative group $\chi^\svee(\Gm)\subset G^\vee$ acts
on the $n$-th summand via the $n$-th power of the tautological character.

Using  \eqref{Hisom} together with the isomorphism
$\sH\cong K^{G^\vee\times \Gm}(\St)$ we can associate to
$\CG\in D^b(\Coh^{\LG\times \Gm}(\St))$ an element $\tau(\CG)\in H$.
Now, for $\CE\in \Coh^{G^\vee\times \Gm}(Z_\rho)$ we set $\varsigma(\CE)=\suml_{n\in \Zet} \tau(\pr_*(\CE)_n)$.
It is easy to see that for every $d$ we have $\varsigma(\CE)_n|_{G(\sF)_d}=0$ for almost all $n$,
so $\varsigma(\CE)\in \widehat{H}$ is well defined. This defines a map  from $\sH_\rho^\spec$ to
$\widehat{H}$, it clearly factors through $H_\rho^\spec$.

It is easy to check that the resulting map $\varsigma\colon  H_\rho^\spec\to \widehat{H}$ is injective;
we will identify $H_\rho^\spec$ with its image under $\varsigma$.

\begin{thm}\label{Th1}
The subspaces  $H_\rho^\spec$, $H_\rho$ in $\widehat{H}$ coincide.
\end{thm}

Theorem \ref{Th1} will be deduced from the following stronger statement.

Let again $\F=\IC_w$ and consider the collection of objects $\F_{(n)}\in D_I(\Fl)$. Using the definition of $\F_{(n)}$ it is not hard to define a collection of morphisms
$\sZ(V)*\F_{(n)}\to \F_{(n+1)}$ which yield an action of the algebra $S_V=\oplusl_n \sZ(\Sym^n(V))$ in the (ind completion of) the monoidal category $D_I(\Fl)$
on the (ind) object $\oplusl_n \F_{(n)}$. Applying the equivalence of \cite{BHecke} to that data we get an object $\oplusl \Phi(\F_{(n)})\in \QCoh^{\LG}(\St)$ together with
an action of the algebra $\Sym(V)\otimes \O_\St$ in $\QCoh^\LG(\St)$.
We will see below that this datum admits a natural lift to an object in $D^b\Coh^\LG (\St\times V^*)$,
we denote that object by $\CE^w$.

\begin{rem}\label{rem_costalk}
Construction of the   morphisms
$\sZ(V)*\F_{(n)}\to \F_{(n+1)}$  yielding an action of  $S_V$
requires the specific definition of $\F_{(n)}$ as in \eqref{Fndef}. From some perspective
the Verdier dual objects $\XF_{(n)} =(\imath^*\jmath_{!*}(\F^\dagger_n))^{S_n}[n]$ introduced
in~\eqref{XF_n} are more natural, but the structure they give rise to is that of a comodule over the
coalgebra dual to $S_V$. That choice forces us to work with costalks and arithmetic Frobenii
when associating a function to a sheaf. While somewhat nonstandard, that way to produce a function
from an $\ell$-adic sheaf has appeared in the literature, see, in particular, \cite[\S A.1.2]{DrW}.
\end{rem}

\begin{thm}\label{Th2}
a) The object $\CE^w$ is (set theoretically)
supported on $Z_\rho\subset \St\times V^*$.

b) The objects $\CE^w$ generate $D^b\Coh_{Z_\rho}(\St\times V^*)$ as a triangulated category. Their classes
freely generate the Grothendieck group.
\end{thm}

Theorem \ref{Th2} shows that the elements $\varsigma(\CE^w)$ span $H_\rho^{spec}$. On the other hand,
it is clear from the definitions that $\varsigma(\CE^w)=f^w$, thus \ref{Th2} implies \ref{Th1}.
The rest of this paper is dedicated to the proof of Theorem \ref{Th2}.

\begin{rem}
  In fact, $\CE^w$ is the direct image under closed embedding of a naturally defined coherent complex on
  the {\em derived scheme} $Z_\rho$. This suggests
that the derived coherent category of that scheme is a natural categorification of $H_\rho$.
\end{rem}


In section~\ref{BD} we recall necessary facts about the Beilinson-Drinfeld global Grassmannian and
its ramified version, explain details of the definition of $\bF$ and describe it  explicitly without
making a direct reference to the global Grassmannian (but using the central sheaves
whose definition relies on it). In section \ref{proofs_RB} we use some properties of the equivalence
between the constructible and the coherent categorifications
of $H$ \cite{BHecke} 
 to prove Theorem \ref{Th2}.

\section{Intersection cohomology costalks on the global affine Grassmannians}
\label{BD}

\subsection{Global affine Grassmannians}
\label{global}
Let $C$ be a smooth curve over $\ol\BF_q$ defined over $\BF_q$. Let $x_0\in C$ be a point
defined over $\BF_q$. For example, we can (and will) take $C=\BA^1\ni0=x_0$.

Recall the morphism of moduli spaces (ind-schemes)
$\mu\colon\ul\Gr_\CG(\ul\sx,\ul\sy)\to\ul\Gr_\CG(\ul\sx\cup\ul\sy)$ introduced in~\cite[3.5.2]{ar}.
Taking $n+1$ points instead of two, we define the proper morphism of ind-schemes
$\mu\colon\ul\Gr_\CG(\ul\sx{}_0,\ul\sx{}_1,\ldots,\ul\sx{}_n)\to
\ul\Gr_\CG(\ul\sx{}_0\cup\ul\sx{}_1\cup\ldots\cup\ul\sx{}_n)$. Note that Achar and Riche work over
the base field $\BC$, but the same definition works over $\ol\BF_q$ (cf.~\cite[5.3]{ar}), and
we will work in this setup. So $\ul\Gr_\CG(\ul\sx{}_0\cup\ul\sx{}_1\cup\ldots\cup\ul\sx{}_n)$
and the other ind-schemes introduced below are $\ol\BF_q$-ind-schemes defined over $\BF_q$.

By construction, $\ul\Gr_\CG(\ul\sx{}_0\cup\ul\sx{}_1\cup\ldots\cup\ul\sx{}_n)$ is equipped with
a projection $\pi$ to $C^{n+1}$. We have a closed embedding $C^n\hookrightarrow C^{n+1},\
(x_1,\ldots,x_n)\mapsto(x_0,x_1,\ldots,x_n)$. We define the {\em ramified Beilinson-Drinfeld
  Grassmannian} $\Gr_n$ as the preimage $\pi^{-1}(C^n)$. We also define the {\em ramified
  Beilinson-Drinfeld convolution diagram} $\wt\Gr_n$ as the preimage $\mu^{-1}(\Gr_n)$.
So we have a proper morphism $\mu\colon\wt\Gr_n\to\Gr_n$.

Following~\cite[3.5.5]{ar}, we denote by $C^{n\dagger}\subset(C\setminus\{x_0\})^n\subset C^n$
the open subset formed by all the multiplicity free configurations of points distinct from $x_0$.
Similarly to~\cite[3.5.5.5]{ar}, we have a canonical isomorphism
$\Gr_n|_{C^{n\dagger}}\cong\Fl\times\Gr_G^n\times C^{n\dagger}$ (recall that $C=\BA^1$).
Similarly to~\cite[3.5.5.4]{ar}, the special fiber
$\Gr_n|_{n\cdot x_0}$ is canonically isomorphic to $\Fl$.
The open embedding $\pi^{-1}(C^{n\dagger})\hookrightarrow\Gr_n$ is denoted by $\jmath$.
The closed embedding $\pi^{-1}(n\cdot x_0)\hookrightarrow\Gr_n$ is denoted by $\imath$.

More generally, for any subset $J\subset\{1,\ldots, n\}$ we denote by $C^n_J$ a locally closed
subvariety $\prod_{i\not\in J}(C_i\setminus\{x_0\})\times\prod_{j\in J}(x_0)_j\subset C^n$ of $C^n$.
The open subvariety formed by all multiplicity free configurations in $C^n_J$ (i.e.
$x_i\ne x_{i'}$ for $i,i'\not\in J,\ i\ne i'$) is denoted by $C^{n\dagger}_J$. The open embedding
$\pi^{-1}(C^{n\dagger}_J)\hookrightarrow\pi^{-1}(C^n_J)$ is denoted by $u^J$.

We will denote the locally closed subvariety $\pi^{-1}(C^n_J)$ of $\Gr_n$ by $Y_J$ for short.
The closed embedding $\ol{Y}\!_J\hookrightarrow Y=\Gr_n$ is denoted by $\imath_J$. In particular,
$\imath=\imath_{\{1,\ldots,n\}}$.

\subsection{Corestriction}
Recall that $\sfS\colon \Rep(\LG)\to \Perv_{G_O}(\Gr_G)$ stands for the geometric Satake equivalence
and $\sZ\colon \Rep(\LG)\to \Perv_I(\Fl)$ stands for its composition with Gaitsgory's central functor,
see e.g.~\cite[2.4.5, 5.3.4]{ar}. An irreducible perverse sheaf $\CF=\IC_w\in\Perv_I(\Fl)$ has a
canonical structure of a mixed sheaf pure of weight 0.
For $V\in\Rep(\LG),\ \sfS(V)$ also has a canonical structure of a mixed sheaf pure of weight 0.
We consider an irreducible perverse sheaf
$\CF^\dagger_n:=\CF\boxtimes\sfS(V)^{\boxtimes n}\boxtimes\ol\BQ_\ell[n](n/2)$
(pure of weight 0) on $\Fl\times\Gr_G^n\times C^{n\dagger}\cong\Gr_n|_{C^{n\dagger}}$.
We set $\CF_n:=\imath^!\wt\CF_n$, where $\wt\CF_n:=\jmath_{!*}\CF^\dagger_n$.

We will compute $\CF_n$ in terms of the mixed central sheaf $\sZ V$ and the nilpotent logarithm
of monodromy operator $s\colon\sZ V\to\sZ V(-1)$. Also we will need the following symmetrized version.
The symmetric group $S_n$ acts in an evident way on $C^n,\ \Gr_n$, and $\wt\CF_n$ carries an evident
equivariant structure. Hence $S_n$ acts on $\imath^!\wt\CF_n$, and we set
$\CF_{(n)}:=(\imath^!\wt\CF_n)^{S_n}[n]$. The main result of~\S\ref{BD} used crucially
in~\S\ref{proofs_RB} for the proof of~Theorem~\ref{Th2} is the following property of
$\oplus_n\CF_{(n)}$~(Corollary~\ref{free support} below): $\oplus_n\CF_{(n)}$ is canonically
quasiisomorphic to a certain Koszul complex that acquires a natural action
of $\Sym^\bullet\sZ V$; moreover, it is a free $\Sym^\bullet\sZ V$-module with cohomology
supported on $\Ker s$.

We start with the computation of $\CF_n$. Namely, let $t_1,\ldots,t_n$ be the coordinate
functions on $C^n=\BA^n$. They can (and will) be viewed as functions on $\Gr_n$. The zero divisor
of $t_i$ in $\Gr_n$ is identified with $\Gr_{n-1}$.
The unipotent nearby cycles functor $\psi_{t_i}\colon \Perv(\Gr_n)\to\Perv(\Gr_{n-1})$ is denoted by
$\psi_i$ for short.

\begin{prop}
  \label{nadler}
  \textup{(a)} For any permutation $\chi\in\fS_n$ we have a canonical isomorphism
  \[\psi_n\psi_{n-1}\cdots\psi_2\psi_1\wt\CF_n\cong
  \psi_{\chi_n}\psi_{\chi_{n-1}}\cdots\psi_{\chi_2}\psi_{\chi_1}\wt\CF_n\in\Perv_I(\Gr_0)=\Perv_I(\Fl).\]
  More precisely, inside the category $\Perv_I(\Fl)$ there is a contractible groupoid whose objects
  include all the possible iterations
  $\psi_{\chi_n}\psi_{\chi_{n-1}}\cdots\psi_{\chi_2}\psi_{\chi_1}\wt\CF_n$.
  The canonically defined object of $\Perv_I(\Fl)$ is denoted by $\psi\wt\CF_n$.
  It is equipped with commuting nilpotent logarithms of monodromy operators
  $s_1,\ldots,s_n\colon\psi\wt\CF_n\to\psi\wt\CF_n(-1)$.

  \textup{(b)} We have a canonical isomorphism $\psi\wt\CF_n\cong\CF\star(\sZ V)^{\star n}$, and
  $s_i$ coincides with the action of $s$ along the $i$-th copy of $\sZ V$.
\end{prop}

\begin{proof}
  (a) For $n=2$ the proof is given in~\cite[\S3.5.7]{ar}. For arbitrary $n$
  see~\cite[Proposition 2.16, Theorem 3.2, Remarks 3.4,3.11]{ar2}.

  (b) For $n=2$ the proof is given in~~\cite[Proposition 3.2.1]{ar}. For arbitrary $n$
  see~\cite[Theorem 3.2, Remark 3.4]{ar2}.
\end{proof}

\begin{rem}
  A particular case of the proposition is proved in~\cite[Lemma 4.4]{s}. Also, it is shown
  in~\cite[Corollaries 16, 17]{s2} that the iterated nearby cycles functor
  in~Proposition~\ref{nadler} is isomorphic to the higher dimensional nearby cycles functor
  of~\cite{o}.
\end{rem}

More generally, given two disjoint subsets
$J=\{j_1<\ldots<j_p\},\ J'=\{j'_1<\ldots<j'_r\}\subset\{1,\ldots,n\}$ and any shuffle
permutation $\chi$ of $J\sqcup J'$, we consider the natural morphism
\[\varkappa_\chi\colon\psi_{j_p}\cdots\psi_{j_1}\imath^*_{J'}\wt\CF_n\to
\varepsilon_{\chi_{j_p}}\cdots\varepsilon_{\chi_{j_1}}
\varepsilon_{\chi_{j'_r}}\cdots\varepsilon_{\chi_{j'_1}}\wt\CF_n,\]
where $\varepsilon_{\chi_j}=\psi_{\chi_j}$ if $\chi_j\in J$, and
$\varepsilon_{\chi_j}=\imath_{\chi_j}^*$ if $\chi_j\in J'$.

\begin{prop}
  The morphism $\varkappa_\chi$ is an isomorphism.
\end{prop}

\begin{proof}
Same as of~Proposition~\ref{nadler}(a).
\end{proof}

The isomorphism $\varkappa_\chi$ together with the isomorphism of~Proposition~\ref{nadler}
gives rise to the morphism
\[\kappa_\chi\colon\psi_{j_p}\cdots\psi_{j_1}\phi_{j'_r}\cdots\phi_{j'_1}\wt\CF_n\to
\varepsilon_{\chi_{j_p}}\cdots\varepsilon_{\chi_{j_1}}
\varepsilon_{\chi_{j'_r}}\cdots\varepsilon_{\chi_{j'_1}}\wt\CF_n,\]
where $\varepsilon_{\chi_j}=\psi_{\chi_j}$ if $\chi_j\in J$, and
$\varepsilon_{\chi_j}=\phi_{\chi_j}$ if $\chi_j\in J'$. Here $\phi_j$ stands for the
(unipotent) vanishing cycles with respect to the function $t_j$.

\begin{cor}
  \label{psiphi}
  The morphism $\kappa_\chi$ is an isomorphism. The resulting perverse sheaf on
  $\ol{Y}\!_{J\sqcup J'}$ will be denoted $\psi_J\phi_{J'}\wt\CF_n$. \hfill $\Box$
\end{cor}

Recall the Beilinson maximal extension functor $\Xi_i:=\Xi_{t_i}$.
Similarly, we have the following

\begin{cor}
  \label{psiphixi}
  Given three disjoint subsets
$J=\{j_1<\ldots<j_p\},\ J'=\{j'_1<\ldots<j'_r\},\ J''=\{j''_1<\ldots<j''_s\}\subset\{1,\ldots,n\}$
  and any shuffle permutation $\chi$ of $J\sqcup J'\sqcup J''$, there is an isomorphism
\[\psi_{j_p}\cdots\psi_{j_1}\phi_{j'_r}\cdots\phi_{j'_1}\Xi_{j''_s}\cdots\Xi_{j''_1}\wt\CF_n\iso
\varepsilon_{\chi_{j_p}}\cdots\varepsilon_{\chi_{j_1}}
\varepsilon_{\chi_{j'_r}}\cdots\varepsilon_{\chi_{j'_1}}
\varepsilon_{\chi_{j''_s}}\cdots\varepsilon_{\chi_{j''_1}}\wt\CF_n,\]
where $\varepsilon_{\chi_j}=\psi_{\chi_j}$ if $\chi_j\in J$, and
$\varepsilon_{\chi_j}=\phi_{\chi_j}$ if $\chi_j\in J'$, while
$\varepsilon_{\chi_j}=\Xi_{\chi_j}$ if $\chi_j\in J''$. The resulting perverse sheaf on
  $\ol{Y}\!_{J\sqcup J'\sqcup J''}$ will be denoted $\psi_J\phi_{J'}\Xi_{J''}\wt\CF_n$.
\hfill $\Box$
\end{cor}

%

Now we consider a graded vector space $\fC^\bullet$ with basis $e_J,\ J\subset\{1,\ldots,n\}$,
where $\deg e_J=n-\sharp J$. We define a complex of perverse sheaves in $\Perv_I(\Fl)$ as follows:
\begin{equation}
  \label{star}
  _n\sC_*^\bullet:=\psi\wt\CF_n(-n)\otimes\fC^\bullet=\CF\star(\sZ V)^{\star n}(-n)\otimes\fC^\bullet,
\end{equation}
with differential
\[d(f\otimes e_{j_1<\ldots<j_{n-p}})
=\sum_{1\leq r\leq n-p}(-1)^rf\otimes e_{j_1<\ldots<j_{r-1}<j_{r+1}<\ldots<j_{n-p}},\ _n\sC_*^p\to{}_n\sC_*^{p+1}.\]
This is nothing but the Chevalley complex of commutative Lie algebra $\fs$ with basis
$e_1,\ldots,e_n$ acting on $\wt\CF_n(-n)$ so that each $e_r$ acts by $\Id$.
Alternatively, this is the total complex of an $n$-dimensional commutative cubical diagram with
$\psi\wt\CF_n(-n)$ standing at all the vertices, and the identity morphisms along all the edges.

We also define another complex of perverse sheaves in $\Perv_I(\Fl)$ as follows:
\begin{equation}
  \label{shriek}
  _n\sC_!^\bullet:=\bigoplus_{i=0}^n\psi\wt\CF_n(-i)\otimes\fC^i
  =\bigoplus_{i=0}^n\CF\star(\sZ V)^{\star n}(-i)\otimes\fC^i,
\end{equation}
with differential
\[d(f\otimes e_{j_1<\ldots<j_{n-p}})
=\sum_{1\leq r\leq n-p}(-1)^rs_rf\otimes e_{j_1<\ldots<j_{r-1}<j_{r+1}<\ldots<j_{n-p}},\ _n\sC_!^p\to{}_n\sC_!^{p+1}.\]
This is nothing but the Chevalley complex of commutative Lie algebra $\fs$ with basis
$e_1,\ldots,e_n$ acting on $\wt\CF_n$ so that each $e_r$ acts by $s_r$ (up to Tate twist).
Alternatively, this is the total complex of an $n$-dimensional commutative cubical diagram with
$\psi\wt\CF_n(-i)$ standing at the vertices, and the logarithm of monodromy morphisms along
all the edges.

Finally, we define a complex of perverse sheaves in $\Perv_I(\Fl)$ as follows:
\begin{equation}
  \label{shriek star}
  _n\sC_{!*}^\bullet:=\on{Im}(\sfs\colon{}_n\sC_!^\bullet\to{}_n\sC_*^\bullet),\
  \sfs(f\otimes e_{j_1<\ldots<j_{n-p}})=s_{j_1}\cdots s_{j_{n-p}}f\otimes e_{j_1<\ldots<j_{n-p}}.
  \end{equation}

\begin{thm}
  \label{calcul}
  There is a canonical isomorphism $\CF_n:=\imath^!\wt\CF_n\cong{}_n\sC_{!*}^\bullet$ in $D^b_I(\Fl)$.
\end{thm}

\begin{proof}
  Recall the locally closed subsets $Y_J\subset Y=\Gr_n$ introduced in the end of~\S\ref{global}.
  Given a collection of perverse sheaves $\CM_J$ on $Y_J$ and a collection of morphisms
  \[\psi_j\CM_{J\setminus j}\xrightarrow{\can}\CM_J\xrightarrow{\var}\psi_j\CM_{J\setminus j}(-1)\]
  for any element $j\in J\subset\{1,\ldots,n\}$, such that the composition is equal to
  $s_j\colon\psi_j\CM_{J\setminus j}\to\psi_j\CM_{J\setminus j}(-1)$, we impose the following
  commutativity relations.

  (a) For any $J\subset\{1,\ldots,n\}$ of cardinality $p$, the $p$-dimensional cubical diagram
  with $\psi_{j'_r}\cdots\psi_{j'_1}\CM_{J\setminus J'}$ at vertices (where
  $J'=\{j'_1<\ldots<j'_r\}\subset J$) and canonical ($\can$) morphisms at edges commutes.

  (b) For any $J\subset\{1,\ldots,n\}$ of cardinality $p$, the $p$-dimensional cubical diagram
  with $\psi_{j'_r}\cdots\psi_{j'_r}\CM_{J\setminus J'}(-\sharp J')$ at vertices (where
  $J'=\{j'_1<\ldots<j'_r\}\subset J$) and variation ($\var$) morphisms at edges commutes.

  Finally, given a collection of perverse sheaves $\CM_J$ on $Y_J$ and
  a collection of morphisms as above, satisfying the commutativity relations (a,b),
  we can iterate Beilinson's gluing construction~\cite[Proposition 3.1]{bei} to obtain
  a sheaf $\CM=\sG(\CM_J,\can,\var)\in\Perv(\Gr_n)$. Conversely, by~Corollary~\ref{psiphi},
  for $\CM=\wt\CF_n$ we have
  $\CM_J=\phi_{j_p}\cdots\phi_{j_1}\wt\CF_n|_{Y_J}$ for $J=\{j_1<\ldots<j_p\}$.

  \begin{lem}
    \label{selfref}
    Let us consider the following gluing data:
    \begin{multline*}
      \CM_J=\on{Im}(s_J\colon \psi_J\wt\CF_n\to\psi_J\wt\CF_n(-\sharp J)=\\
      u^J_{!*}\left(\on{Im}\big(s_J\colon\CF\star(\sZ V)^{\star J}\to
      \CF\star(\sZ V)^{\star J}(-\sharp J)\big)\boxtimes
    \sfS(V)^{\boxtimes\{1,\ldots,n\}\setminus J}\boxtimes\ol\BQ_\ell[n-\sharp J]
    \Big(\frac{n-\sharp J}{2}\Big)\right)
    \end{multline*}
    (the open embedding $u^J\colon\pi^{-1}(C^{n\dagger}_J)\hookrightarrow\pi^{-1}(C^n_J)$ was
    introduced in the end of~\S\ref{global}. Note that $\pi^{-1}(C^{n\dagger}_J)$ is naturally
    isomorphic to $\Fl\times\Gr_G^{\{1,\ldots,n\}\setminus J}\times C_J^{n\dagger}$).

    The morphism $\can\colon\psi_j\CM_{J\setminus j}\to\CM_J$ is induced by $s_j$, and the
    morphism $\var\colon\CM_J\to\psi_j\CM_{J\setminus j}(-1)$ is induced by the embedding
    of the image of $s_j$ into the target.

    Then $\sG(\CM_J,\can,\var)\cong\wt\CF_n$.
  \end{lem}

  \begin{proof}
    Iterating Beilinson's gluing construction, we see that $\sG(\CM_J,\can,\var)$ is the middle
    cohomology of the total complex $\sG^\bullet$ of the following $n$-dimensional commutative
    cubical diagram.
    Its entries are numbered by partitions $\{1,\ldots,n\}=J'\sqcup M\sqcup J$, and an entry
    \[\sG_{J',M,J}:=\bigoplus_{M'\subset M}\psi_J\Xi_{M'}\psi_J\CM_{M\setminus M'}(-\sharp J'),\]
    where $\psi_J\Xi_{M'}\psi_J\CM_{M\setminus M'}$ stands for the ``normally ordered'' composition
    of nearby cycles and maximal extensions: the operations with smaller indices are applied first.
    For example, $\psi_{\{1,3\}}\Xi_{\{2,4\}}\psi_{\{5\}}:=\psi_5\Xi_4\psi_3\Xi_2\psi_1$. The morphisms
    along the edges of our $n$-dimensional cube are induced by $\can$ and $\var$.

    We have to check that the restrictions (resp.\ corestrictions) of $\sG(\CM_J,\can,\var)$ to
    $Y_J$ live in perverse negative (resp.\ positive) degrees unless $J=\emptyset$. By induction
    in $n$ it suffices to check this for $J=\{1,\ldots,n\}$. We consider the corestriction to
    $Y_{\{1,\ldots,n\}}$, and the argument for the restriction is similar.

    By induction in $n$, for $J=\{j_1<\ldots<j_p\}\subsetneqq\{1,\ldots,n\}$, our $\CM_J$ coincides
    with $\phi_{j_p}\cdots\phi_{j_1}\wt\CF_n|_{Y_J}$. By~Corollary~\ref{psiphixi}, the most part of
    $4^n$ summands of $\imath^!\sG^\bullet$ cancel out, and we are left with the total complex
$\ol\sG{}^\bullet$ (living in degrees $0,\ldots,n$) of the $n$-dimensional cubical diagram with entries
    $\ol\sG_J:=\psi_{\{1,\ldots,n\}\setminus J}\CM_J(\sharp J-n)$ at vertices and the morphisms induced
    by $\var$ along the edges. In particular, in degree~0, the only summand is $\CM_{\{1,\ldots,n\}}$,
    and the differential to the degree~1 component is injective, so
$\imath^!\sG^\bullet\cong\ol\sG{}^\bullet$ vanishes in degree~0 and lives in strictly positive degrees.
    \end{proof}

  Note that the complex $\ol\sG{}^\bullet$ computing $\imath^!\wt\CF_n$ by the proof
  of~Lemma~\ref{selfref}, coincides with $_n\sC_{!*}^\bullet$ of~\eqref{shriek star}. So the
  theorem is proved.
\end{proof}

\subsection{Symmetrized version}
The symmetric group $S_n$ acts in an evident way on $C^n,\ \Gr_n$, and $\wt\CF_n$ carries an evident
equivariant structure. Hence $S_n$ acts on $\imath^!\wt\CF_n$, and we set
$\CF_{(n)}:=(\imath^!\wt\CF_n)^{S_n}[n]$. We deduce from~Theorem~\ref{calcul} the following description
of $\oplus_n\CF_{(n)}$.

We consider the following 2-term complexes sitting in degrees $-1,0$:
\[\CH_1=\big(\sZ V(-1)\xrightarrow{\Id}\sZ V(-1)\big),\
\CH_s=\big(\sZ V\xrightarrow{s}\sZ V(-1)\big),\]
and a morphism $\sfs\colon \CH_s\to\CH_1$:
\[\begin{CD}
\sZ V @>>s> \sZ V(-1)\\
@VVsV @V{\Id}VV\\
\sZ V(-1) @>{\Id}>> \sZ V(-1).
\end{CD}\]
It induces the same named morphism
$\sfs\colon \sK_s^\bullet:=\Sym^\bullet\CH_s\to\Sym^\bullet\CH_1=:\sK_1^\bullet$
of Koszul complexes. Finally, we consider the morphism
$\Id_\CF\star\sfs\colon\CF\star\sK_s^\bullet\to\CF\star\sK_1^\bullet$. We denote the image of this
morphism by $\sK_{s1}^\bullet$ (a $-\BN$-graded complex of perverse sheaves on $\Fl$).
Then~Theorem~\ref{calcul} implies the following

\begin{cor}
  \label{calcul sym}
  There is a canonical isomorphism $\oplus_n\CF_{(n)}\cong\sK_{s1}^\bullet$. \hfill $\Box$
  \end{cor}

Since $\Id_\CF\star\sfs\colon\CF\star\sK_s^\bullet\to\CF\star\sK_1^\bullet$ is induced by the
homomorphism $\sfs$ of dg-algebras $\Sym^\bullet\CH_s\to\Sym^\bullet\CH_1$ that are both
$\Sym^\bullet\sZ V$-modules, its image acquires a structure of $\Sym^\bullet\sZ V$-module.

\begin{cor}
  \label{free support}
  $\oplus_n\CF_{(n)}\cong\sK_{s1}^\bullet$ is a free $\Sym^\bullet\sZ V$-module with cohomology supported
  on $\Ker s$.
\end{cor}

\begin{proof}
  First, the cohomology of $\Sym^\bullet\CH_s$ is supported on $\Ker s$, so the cohomology of the
  image of
  \[\Id_\CF\star\sfs\colon\CF\star\Sym{}\!^\bullet\CH_s\to\CF\star\Sym{}\!^\bullet\CH_1\]
  is supported on $\Ker s$ as well. Furthermore,
  \begin{multline*}
    \on{Im}\big(\Id_\CF\star\sfs\colon\CF\star\Sym{}\!^\bullet\CH_s\to\CF\star\Sym{}\!^\bullet\CH_1\big)\\
    =\on{Im}\Big(\Id_\CF\star\Sym{}\!^\bullet(s[1])\colon\CF\star\Sym{}\!^\bullet(\sZ V[1])\to
    \CF\star\Sym{}\!^\bullet(\sZ V[1])\Big)\star\Sym{}\!^\bullet\sZ V.
  \end{multline*}
\end{proof}

\section{Proof of Theorem \ref{Th2}}
\label{proofs_RB}

\subsection{Recollection on the coherent description of the affine Hecke category}
We start by recalling the equivalence between two categorifications of the affine Hecke algebra.

Consider the Steinberg variety of triples $\St=\widetilde{\Lg}\times _{\Lg} \widetilde{\N}$ for $\Lg$ and
 the affine flag variety $\Fl$ of $G$.
Let  $\P=\Perv_{I^0}(\Fl)$ be the category of perverse sheaves on equivariant with respect to the radical $I^0$ of an Iwahori group $I$.

\begin{thm} \cite{BHecke} \label{Two_cat}
a) There exists a natural equivalence
\begin{equation}\label{Haff}
D^b(\Coh^\LG(\St))\cong D^b(\P).
\end{equation}

For a representation $V$ of $\LG$ we have:

 b)  The equivalence \eqref{Haff}  intertwines the functor  $\F\mapsto F\otimes V$
 with the functor of convolution with $\sZ(V)$.

 c) Let $\bar{\tau}_V$ be the canonical endomorphism of $\O_{\Lg}\otimes V$ such that the action of $\tau_V$ on the fiber at $x$
 equals the action of $x$ on $V$; let $\tau_V$ be the pull back of $\bar{\tau}_V$ to an endomorphism of $\O_{\St}\otimes V$.
 Then \eqref{Haff} intertwines $\tau_V$ with the endomorphism induced by log monodromy.
 \end{thm}

\begin{rem}
There are also versions of the equivalence involving monoidal categories, a natural categorification
of the affine Hecke algebra. We chose to work with the present non-monoidal one for the sake of a technical simplification.
\end{rem}

Next, we recall an explicit description
of the $t$-structure on the left hand side of \eqref{Haff} corresponding to
 the natural $t$-structure on the right hand side.

 In \cite{BM} we defined the so called noncommutative Springer resolution.
 It is a $\sk$-algebra $A$  with a derived
 equivalence $D^b(A\modu_{\on{fg}})\cong D^b(\Coh(\widetilde{\Lg}))$
 (here $\sk=\overline\BQ_\ell$). We let $\Ap=A\otimes_{\Sym(\ft)}\sk$, then we also have:
\[D^b(\Ap\modu_{\on{fg}})\cong D^b(\Coh(\widetilde\N));\]
\begin{equation}\label{AA}
D^b(\bA\modu_{\on{fg}}^{\LG}) \cong D^b(\Coh^{\LG}(\St)),
\end{equation}
where $\bA=A\otimes_{\O(\Lg)} \Ap^{op}$.

Moreover, the composition of \eqref{AA} and \eqref{Haff}
 is an equivalence which satisfies the following property. Consider
the filtration on $D^b(\bA\modu_{\on{fg}}^{\LG})$ by  support as a module over the center;
this is a filtration indexed by the poset of nilpotent orbits in $\Lg$. Then both $t$-structures
are compatible with this filtration, and the induced $t$-structures on the associated graded category corresponding to
an orbit $O$  differ by shift by $\frac{1}{2}{\mathrm{codim}}(O)$, see \cite[Theorem 54]{BHecke}
corrected in \cite{BL}.

In other words we have the following

\begin{thm}\label{pervAmod} \cite{BHecke} \textup{a)} The category $D^b(\bA\modu_{\on{fg}}^{\LG})$
carries a unique (bounded) $t$-structure such that the forgetful
functor $D^b(\bA\modu_{\on{fg}}^{\LG})\to D^b(\Coh^\LG(\N))$
is $t$-exact where the target category is equipped with the perverse coherent
$t$-structure of middle perversity.

\textup{b)} The tautological $t$-structure on $D^b(\P)$ corresponds
to the $t$-structure from part (a) under the composition
of equivalences.
\end{thm}

Let $\bA\modu^\LG_{\on{perv}}$ denote the heart of the $t$-structure from Theorem \ref{pervAmod}a),
its objects will be called perverse $\bA$-modules. Recall that an irreducible equivariant perverse coherent sheaf
of middle perversity can be described as an intermediate extension of an irreducible equivariant vector bundle from an
orbit, see \cite{AB}. Thus we obtain

\begin{cor}
The above equivalence sends irreducible objects $\IC_w$ in $\P$
to objects of the form $j_{!*}^O(\EE_L[d_w])$ where $O=\LG(e)$ is a nilpotent orbit,  $L$
is an irreducible $Z(e)$-equivariant $\bA_e$-module and $\EE_L$ is the corresponding
equivariant module for $\bA|_O$ and $d_w=\frac{\codim_\N(O_w)}{2}$ (here $e\in O$ is an element, $Z(e)$ is its centralizer and
$\bA_e=\bA\otimes _{\O(\N)}\sk_e$ is the  quotient of $\bA$ by the corresponding maximal
ideal in the center).
It induces a bijection between the two classes of objects.
\end{cor}

We denote the $(O_w,L_w)$ the orbit and the equivariant module corresponding to $\IC_w$.

\subsection{End of proof}
Set $\bA_V=\Sym(V)\otimes_\sk \bA$.
As an immediate consequence of \eqref{AA} we get:

\begin{equation}\label{AAV}
D^b(\bA_V\modu_{\on{fg}}^{\LG}) \cong D^b(\Coh^{\LG}(\St\times V^*)).
\end{equation}

By a graded $\Sym(V)$-module in $\bA\modu^\LG_{\on{perv}}$ we mean
a collection of objects $\CG_n\in \bA\modu^\LG_{\on{perv}},\ n\in \Zet$, together with morphisms
$V\otimes\CG_n\to\CG_{n+1}$ satisfying the obvious compatibility.

Such a module is called
free if it has the form $\CG_n=\Sym^{n-d}(V)\otimes\CG$ for some
$d\in\BZ,\ \CG\in\bA\modu^\LG_{\on{perv}}$.

\begin{lem}
  \label{real}
  Let $\Hop$ denote the homotopy category of finite complexes of free graded  $\Sym(V)$-modules
in $\bA\modu_{\on{perv}}$.

We have a natural realization functor
\[\Hop \to D^b(\bA_V\modu_{\on{fg}}^\LG)\cong D^b(\Coh^{\LG\times \Gm}(\St\times V^*)).\]
\end{lem}

\proof Follows from the standard application of filtered derived categories, see~\cite{bei1}. \qed

Recall the object $\cE^w\in D^b\Coh^{G^\vee}(\St\times V^*)$ introduced right before~Theorem~\ref{Th2}.
In view of Lemma~\ref{real} we can upgrade  $\cE^w$ to an object of  $D^b(\bA_V\modu_{\on{fg}}^\LG)$.

We are now ready to state a key property of this object which formally implies~Theorem~\ref{Th2}.

For an orbit $O\subset \N$ we set $\partial O=\overline{O} \setminus O$ where $\overline{O}$ is the closure of $O$.

\begin{prop}\label{abc}
\textup{a)} The object $\cE^w$ is (set-theoretically) supported on $Z_\rho$.

\textup{b)}  The object $\cE^w$ is (set-theoretically) supported on the closure of $O_w$.

\textup{c)} The restriction of $\cE^w$ to the open set $\N\setminus \partial O_w$ is isomorphic to
$\EE_{L_w}\otimes \Sym(\EE_{V_e})[d_w]$ where $e\in O$, $V_e$ is the space of coinvariants
of $e$ acting on $V$ and $\EE_{V_e}$ is the corresponding equivariant vector bundle on $O$.
\end{prop}

\proof a) is immediate from Corollary \ref{free support}. b) is also clear from the construction:
the image of $\IC_w$ under \eqref{Haff}, \eqref{AA}  is supported on $\overline{O}_w$ by definition
of the latter, hence so are the images of the terms of the complexes
$\IC_w\star\sK_s^\bullet$, $\IC_w\star\sK_1^\bullet$.

Since the $t$-structure is compatible with the filtration by nilpotent orbits, we see that the image of the map
between the terms of the two complexes is also supported on $\overline{O}_w$.

To prove c) consider the image of the complex   $\sK_{s1}^\bullet$ (constructed from $\F=\IC_w$) under the equivalence
\eqref{AA} restricted to the open part $\N\setminus \partial O_w$ of its support, we denote it by $K_w$.
We provide an explicit description of $K_w$ as follows.
Consider the vector bundle $\EE_{V_e}$ on $O_w$. Its pull back to
 $O_w\times V^*$ carries a tautological section, let
$\KK_{O_w}$ denote the corresponding Koszul complex.

We claim that
\begin{equation}\label{compl_iso}
K_w\cong \KK_{O_w} \otimes \EE_{L_w}[d_w].
\end{equation}

By Theorem \ref{pervAmod}b) the induced equivalence between equivariant $\bA$-modules supported
on the orbit and the corresponding subquotient\footnote{This is the subquotient corresponding
to the 2-sided cell in the affine Weyl group, though we do not use this description here.}  of $D^b(\P)$
is $t$-exact up to the shift by $d_w$.

 Thus the isomorphism \eqref{compl_iso} follows from Theorem \ref{Two_cat}:  the Koszul complex
 $\KK_{O_w}$ is isomorphic to the image of the Koszul complex associated to the tautological
 section $v^*$ of the trivial vector bundle $V^*\otimes \O_{O_w\times V^*}$ mapping to the Koszul
 complex of the same bundle with the section $e(v^*)$ where $e$ is the tautological section
 of the bundle with fiber $\Lg$ on $O_w$.

 Clearly $H^i(\KK_w)=0$ for $i\ne 0$ while $H^0(\KK_w)=\O(Z_\rho\cap \pi^{-1}(O_w))$, which yields~c). \qed

 Theorem \ref{Th2} is a direct corollary of Proposition \ref{abc}.

 \section{Appendix: relation to arcs into semi-groups}
 \label{semi-groups}

\subsection{From semi-groups to representations}\label{sg-to-rep}
Let $\oG$ be a normal affine algebraic semi-group whose group of invertible elements is $G$.
Let also $\chi\colon G\to \GG_m$ be a character which extends to a regular morphism $\oG\to \AA^1$.
Let $\overline{\Lam}$ denote the subset of those coweights $\lam\colon \GG_m\to T\subset G$ which extend
to a morphism $\AA^1\to \oG$. It is easy to see that $\overline{\Lam}$ is a cone inside $\Lam$.
Let us denote by $\Lam_+$ the set of dominant coweights of $G$; set $\overline{\Lam}_+=\Lam_+\cap \overline{\Lam}$.
We say that an element $\lam\in \overline{\Lam}_+$ is indecomposable if it cannot be represented as a sum of two other elements of $\overline{\Lam}_+$. It is easy to see that the set of indecomposable elements is finite. We denote by $S$ the set of indecomposable elements of $\overline{\Lam}_+$ and by $S_{\max}$ the subset of $S$ consisting of those elements of $S$ which are maximal with respect to the standard partial order on $\Lam$.
We now define the  representation $(V,\rho)$ of $G^{\vee}$ to be the direct sum of $V(\lam)$ where
$\lam$ runs over all elements of $S_{\max}$.

\medskip
\noindent
{\bf Example.} Let $\oG$ consist of pairs $(A,x)$ where $A$ is a two-by-two matrix, and $x\in \AA^1$ such that
$\det(A)=x^n$ for some $n>0$. Then $G$ is isomorphic to $\GL(2)$ if $n$ is odd and to $\SL(2)\times \GG_m$ is $n$ is even; so the dual group is either $\GL(2)$ or $\PGL(2)\times \GG_m$.
The lattice $\Lam$ naturally identifies with triples $(a,b,c)\in \ZZ^3$ with $a+b=nc$; the cone $\Lam_+$ consists of triples as above such that $a\geq b$. The cone $\overline{\Lam}_+$ consists of triples as above such $a,b,c\geq 0$.
Then it is easy to see that indecomposable elements are those of the form $(a,b,1)$ with $a+b=n$ and $a\geq b$. Among those the element $(n,0,1)$ is maximal, so $S_{\max}$ consists of one element. The representation $\rho$ is
then isomorphic to the $n$-th symmetric power of the standard representation of $\GL(2)$ (note that when $n$ is even it can naturally be regarded as a representation of $\PGL(2)\times \GG_m$).

\subsection{The conjecture}
Let $\oX$ be an algebraic variety with a smooth open dense subset $X$. Let $\oX_{\sO}$ denote the scheme of formal arcs into $\oX$, i.e. the scheme of maps $\Spec(\sO)\to \oX$. Let $\oX{}_\sO^0$ denote the open sub-scheme of $\oX_{\sO}$ consisting of those maps $\Spec(\sO)\to \oX$ which send $\Spec(\sF)$ to $X$.

Assume that both $X$ and $\oX$ are affine. Then we have an ind-scheme $X_{\sF}$ parametrizing maps $\Spec(\sF)\to X$, and similarly for $\oX$. In addition, let $\oX{}_\sF^0$ denote the open sub-scheme of those maps
$\Spec(\sF)\to \oX$ which land in $X$. We have a natural map of ind-schemes $X_{\sF}\to\oX{}_\sF^0 $.
This map induces a bijection between $\FF_q$-points of $\oX{}_\sF^0$ and $X_{\sF}$; however, scheme-theoretically it is usually far from being an isomorphism; typically $\oX{}_\sF^0$ has a stratification such that on the level of reduced schemes $X_{\sF}$ is isomorphic to the disjoint union of the corresponding strata; in particular, on every connected component of $X_{\sF}$ the above map is a locally closed embedding on the level or reduced schemes.

Ideally, we would like to perform the following procedure: start with an irreducible perverse sheaf
$\CF$ on $X_{\sF}$, think of it as a sheaf on a locally closed subset $\oX{}_\sF^0$ and take its
Goresky-Macpherson extension to $\oX{}_\sF^0$; however, technically, at the moment we do not have
the language to do this (for example, because we don't know how to define the category of perverse
sheaves on $\oX{}_\sF^0$). On the other hand, it turns out, that one can define the corresponding
function on $\oX{}_\sF^0(\FF_q)=X_{\sF}(\FF_q)=X(\sF)$. To simplify the discussion we shall do it in
the case when $\CF$ is an irreducible perverse sheaf on $X_{\sO}$ (which is a closed subset of
$X_{\sK}$); later on we shall apply it to the case when $X=G$ and $\oX=\oG$.

Note that the closure of $X_{\sO}$ in $\oX{}_\sF^0$ is $\oX{}_\sO^0$. According to~\cite{Dr}
and~\cite{GrK} the formal  neighbourhood of every point $x(t)$ of $\oX{}_\sO^0$ is isomorphic to a
(completed) product $\widehat{\AA^{\infty}_0}\times \hatZ_{z_0}$ where:

1) $\AA^{\infty}$ denotes the infinite-dimensional affine space, i.e.\
the affine scheme $\Spec \FF_q[x_1,x_2,\ldots]$;

2) $\widehat{\AA^{\infty}_0}$ is the formal neighbourhood of the point $0$ in $\AA^{\infty}$;

3) $Z$ is a scheme of finite type over $\FF_q$, $z_0\in Z(\FF_q)$ and $\hatZ_{z_0}$ denotes the formal neighbourhood of $z_0$ in $Z$.

\noindent
If $\CF$ is an irreducible perverse sheaf on $X_{\sO}$, then properties 1--3 above allow us to define
canonically the stalk of the would be Goresky-Macpherson extension of $\CF$ to $\oX{}_\sO^0$ at
$x(t)$ (we refer the reader to \cite{BNS,ngo} for details). In particular, we can take the trace of
Frobenius on this stalk and get an element of $\qlb$. So, if we denote the locally closed embedding
$X_{\sO}\to \oX{}_\sO^0$ by $i$, then we can't make sense of $i_{!*}\CF$, but we can make sense of the
function $f_{i_{!*}\CF}\colon\oX{}_\sO^0(\FF_q)\to \qlb$ (which would correspond to $i_{!*}\CF$ by the
Grothendieck sheaf-function correspondence  if $i_{!*}\CF$ were defined).

Let now $a_{\chi}\colon\sK_0(\Rep(G^{\vee}))\to \sK_0(\Rep(G^{\vee}))$ be the map that sends the class $[V]$ of every irreducible representation $V$ of $G^{\vee}$ to $q^{i(\chi,V)}[V]$ where $i(\chi,V)\in \ZZ$ is such that the action of $\GG_m$ on $V$ given by the cocharacter $\chi^{\vee}$ of $Z(G^{\vee})$ is given by the character $z\mapsto z^{i(\chi,V)}$.
\begin{conj}\label{semi}
\begin{enumerate}
\item
Assume that $\CF$ is the constant sheaf on $G_{\sO}$. Then
the function $f_{i_{!*}\CF}$ is equal to
$S(a_{\chi}([\Sym(V)])$ for $\rho$ as in~\S\ref{sg-to-rep} (recall that $S$ denotes the Satake
isomorphism).
\item
  Assume that $(V,\rho)$ is irreducible. For any $n\in \ZZ$ let $G_n$ be the (open and closed) subset
  $v(\chi(g))=n$ of $G_{\sF}$. Assume that $\CF$ is supported on $G_0$ (but equivariance of $\CF$
  is not assumed here). Then
$(f_{i_{!*}\CF})|_{G_n}\simeq q^{n\cdot i(\chi,V)}f_{\CF_{\rho,(n)}}$.
\end{enumerate}
\end{conj}

Let us make several remarks.
First, in case when $\CF$ is the constant sheaf on $G_{\sO}$, Conjecture \ref{semi} is proved
in~\cite{BNS} either when $G$ is  torus, or when $(V,\rho)$ is irreducible. We believe that it
should not be hard to extend the arguments of~\cite{BNS} to arbitrary $\CF$.

Second, it should not be difficult to deduce Conjecture \ref{semi} from the following statement,
which is proved in~\cite{BNS}.
For $n>0$ consider the closed subset of the (usual, not ramified)  Beilinson-Drinfeld Grassmannian
$\Gr_{\oG,n}$ consisting of pairs $(\calF,\kappa)$ where $\CP$ is a principal $G$-bundle on the curve
$C$, and $\kappa$ is an isomorphism between $\CP$ and the trivial bundle $\CP^0$ away from points
$(x_0,x_1,\ldots, x_n)$ (here as before $x_0$ is fixed and the points $x_i$ for $i>0$ move around $C$)
such that the rational composition map $C\to C\times G=\CP^0\to\CP$ extends to a regular map
$C\to\oG\overset{G}\times\CP$.

Let now $(\CP,\kappa)$ be a point of $\Gr_{\oG,n}$ as above such that $\kappa$ is actually an
isomorphism between $\CP$ and $\CP^0$ on $C\setminus\{x_0\}$ (i.e. we have $x_i=x_0$ for all $i$). Now if we choose a trivialization
of $\CP$ on the formal neighbourhood of $x_0$ and also choose a formal parameter $t$ near $x_0$,
then $(\CP,\kappa)$ defines a point $g(t)$ of $\oG{}_\sO^0$.
The following proposition is proved in~\cite[Propositions 2.1,2.2]{BNS} and~\cite[Proposition 11]{bo}.

\begin{prop}
For some integer $m>0$ we have an isomorphism of completed products
\[(\widehat{\oG{}_\sO^0})_{g(t)}\ \times \ \widehat{\AA}^{m}_0\simeq (\widehat{\Gr_{\oG,n}})_{(\CP,\kappa)}\
\times\ \widehat{\AA}^\infty_0.\]
In other words, the scheme $\Gr_{\oG,n}$ (which is of finite type over $\FF_q$) essentially plays the
role of the scheme $Z$ in the Drinfeld-Grinberg-Kazhdan theorem (up to smooth equivalence of formal
neighbourhoods).
\end{prop}

\subsection{Locality} We conclude the paper with yet one more conjecture which roughly speaking states that in the above setup the space $\calS_{\rho}(G(\sF))$ is local on $\oG(\sF)$.
\begin{conj}
There exists a $G(\sF)\times G(\sF)$-equivariant sheaf $\bfS_{\rho}$ of $\CC$-vector spaces on $\oG(\sF)$ such that
\begin{enumerate}
\item
$\bfS_{\rho}|_{G(\sF)}$ is the sheaf of locally constant functions on $G(\sF)$.
\item
The space $\calS_{\rho}(G(\sF))$ is equal to the space of global sections of $\bfS_{\rho}$ on $\oG(\sF)$ with compact support.
\end{enumerate}
\end{conj}

\end{document}